\documentclass[11pt,a4paper]{article}
\usepackage[american]{babel}
\usepackage{a4}
\usepackage{hyperref}
\usepackage{dsfont}
\usepackage{amsmath, amssymb, amsthm}
\usepackage{graphicx}
\usepackage{subfig}
\usepackage{ifthen}
\usepackage{pgf,tikz}
\usepackage{mathrsfs}
\usetikzlibrary{positioning,arrows,patterns}

\definecolor{zzttqq}{rgb}{0.6,0.2,0.}

\newtheorem{theorem}{Theorem}
\newtheorem{proposition}[theorem]{Proposition}
\newtheorem{lemma}[theorem]{Lemma}
\newtheorem{corollary}[theorem]{Corollary}

\theoremstyle{definition}

\newtheorem*{definition}{Definition}
\newtheorem*{example}{Example}
\newtheorem*{remark}{Remark}

\usepackage{overpic}
\usepackage{contour}\contourlength{0.2ex}

\def\veccc#1/#2/#3;{\begin{pmatrix}#1\\#2\\#3\end{pmatrix}}
\def\cput(#1,#2,#3){\put(#1,#2){\contour*{white}{#3}}}%
\def\colorboxput(#1,#2,#3,#4){\put(#1,#2){\textcolor{#3}{\vrule width.6em height.6em depth0pt}~#4}}

\def\vw{{\mathbf v}}

      \def\nw{{\mathbf n}}
   \def\vw{{\mathbf v}}

\def\be{\begin{equation}}   \def\ee{\end{equation}}
\def\cput(#1,#2)#3{\put(#1,#2){\hbox to 0pt{\hss#3\hss}}}

\begin{document}

\title{The Gauss map of polyhedral vertex stars}

\author{Thomas F. Banchoff$^{1}$ \and Felix G\"unther$^{2}$}

\maketitle

\footnotetext[1]{Department of Mathematics, Brown University, Box 1917, 151 Thayer Street, Providence, RI 02912. E-mail: thomas\_banchoff@brown.edu}
\footnotetext[2]{Max Planck Institute for Mathematics, Vivatsgasse 7, 53111 Bonn, Germany. E-mail: fguenth@math.tu-berlin.de}

\begin{abstract}
\noindent
In discrete differential geometry, it is widely believed that the discrete Gaussian curvature of a polyhedral vertex star equals the algebraic area of its Gauss image. However, no complete proof has yet been described. We present an elementary proof in which we compare, for a particular normal vector, its winding numbers around the Gauss image and its antipode with its critical point index. This index is closely related to the normal degree of the Gauss map. We deduce how the number of inflection faces is related to the numbers of positive and negative components in the Gauss image. The resulting formula significantly limits the possible shapes of the Gauss image of a polyhedral vertex star. For example, if a triangulated vertex star is in general position and its Gauss image has no self-intersections, then it is either a convex spherical polygon if the curvature is positive, or a spherical pseudo-quadrilateral if the curvature is negative.\\ \vspace{0.5ex}

\noindent
\textbf{2010 Mathematics Subject Classification:} 52B70.\\ \vspace{0.5ex}

\noindent
\textbf{Keywords:} Discrete differential geometry, polyhedral surface, discrete Gaussian curvature, Theorema Egregium, degree theory.
\end{abstract}

\raggedbottom
\setlength{\parindent}{0pt}
\setlength{\parskip}{1ex}


\section{Introduction}\label{sec:intro}

In differential geometry, the Gaussian curvature at a point of a two-dimensional surface immersed in Euclidean three-space equals the ratio by which the Gauss map scales infinitesimal areas around that point. A standard discretization of Gaussian curvature for polyhedral surfaces is the angle deficit. Even though it is widely believed in discrete differential geometry that the discrete Gaussian curvature of an embedded polyhedral vertex star equals the algebraic area of the image of its Gauss map, the literature seems not to contain a complete proof of this fact. Actually, this question is a popular exercise in some lecture notes on discrete differential geometry. But their idea of proof, which goes back to Alexandrov \cite{A05}, is only valid in the convex setting: They make use of the fact that no inflections face is present to deduce the right formula linking the spherical angles of the Gauss image with the angles at the vertex star. P\'olya raised the question whether this theorem is true for non-convex vertex stars in his note \cite{P54} on the Gauss-Bonnet theorem for (convex) embedded polyhedral disks. In this paper, we give an affirmative answer and give a correct and elementary proof.

In Section~\ref{sec:Gauss}, we present the elementary proof of the main theorem. Our idea bases on the critical point index $i(v,\xi)$ of a normal vector $\xi \in S^2$ with respect to the star of a vertex $v$. The critical point index was first introduced by the first author in \cite{B67,B70}. There, the discrete Gaussian curvature of a vertex was given by \[K(v)=\frac{1}{2}\int\limits_{S^2}i(v,\xi)d\xi.\] It was shown that this term equals the angle deficit. This representation of the discrete Gaussian curvature makes its relation to the algebraic area of the Gauss image apparent as long as the Gauss image does not contain any pair of antipodal points. Still, a formal proof of this correspondence has not yet been given. We complete the discussion of \cite{B67,B70} by giving a proof of the general fact that $i(v,\xi)$ equals the sum of the winding numbers of $\xi$ with respect to the Gauss image and its antipode.

The theorem relating the algebraic area of the Gauss image and the angle deficit generalizes to polyhedral vertex stars that are not embedded. In the general setup, the algebraic area of a spherical curve is a priori defined only up to $\pm 4k\pi$, $k\in\mathds{Z}$. Furthermore, it turns out that an additional summand of $2\pi$ has to be added (or subtracted) to the angle deficit if the number of self-intersections is odd. The distinction between an even and odd number of self-intersections corresponds to the two regular homotopy classes of immersions of a circle into the sphere classified by Smale \cite{S58}. The ambiguity of the summand $\pm 4k\pi$ can be resolved by assigning one normal vector to be outside and allowing only regular homotopies that avoid this vector. The different integers $k$ then correspond to the Whitney-Graustein theorem that classifies regular homotopy classes of immersions of a circle into the plane \cite{W37}. We discuss these topics in Section~\ref{sec:degree} and continue with commenting on the relations between the index and the normal degree. Using deformation techniques, we show that the absolute value of the normal degree is always smaller or equal than the absolute value of the index, assuming that we are in the case of an embedded polyhedral vertex star. Since the critical point index cannot be larger than 1, we deduce that the Gauss image of an embedded vertex star contains at most one component whose area contributes positively to the algebraic area. Actually, the possible shapes of Gauss images are much more restricted as follows from our formula relating the number of inflection faces in a (not necessarily embedded) vertex star with certain numbers of positively and negatively oriented components in the Gauss image. This discussion is carried out in Section~\ref{sec:shape_analysis}.

For example, if the Gauss image of a simplicial vertex star in general position is free of self-intersections, then the only possible shapes are convex spherical polygons in the case of positive discrete Gaussian curvature and spherical pseudo-quadrilaterals in the case of negative curvature. Gauss images that are free of self-intersections are the fundamental idea in the recent discussion of smoothness of polyhedral surfaces by the second author, Jiang, and Pottmann in \cite{GJP16}. Star-shaped Gauss images correspond to vertex stars that admit a nice discrete Dupin indicatrix. Together with Wallner they discussed in \cite{JGWP16} how smoothness of polyhedral surfaces can be implemented into an algorithm easing the design of freeform structures in architecture.


\section{Critical point index and discrete Gaussian curvature}\label{sec:Gauss}

The aim of this section is to show our main Theorem~\ref{th:egregium}. It states that the angle deficit of an embedded polyhedral vertex star equals the algebraic area of its Gauss image. On the way, we present examples that demonstrate the complexity of the problem and the necessity of a careful investigation.

We start with an introduction of the basic notation and a formulation of Theorem~\ref{th:egregium} in Section~\ref{sec:curvature}. We continue with Alexandrov's argument that is valid for a convex vertex star in Section~\ref{sec:Alexandrov}. Critical point indices are introduced in Section~\ref{sec:critical} and compared with winding numbers in Section~\ref{sec:winding}. Our discussion culminates in Section~\ref{sec:proof} in a proof of Theorem~\ref{th:egregium}. 


\subsection{Discrete Gaussian curvature and the main theorem}\label{sec:curvature}

Throughout Section~\ref{sec:Gauss}, let $P$ be a closed polyhedral surface embedded into three-dimensional Euclidean space. We assume that all faces of $P$ are planar triangles and that any two faces sharing an edge are not coplanar, meaning that $P$ is a simplicial surface. We suppose the latter condition merely for the ease of our exposition only. In general, we can approximate a general polyhedral surface with faces of more than three vertices or even non-convex faces by a simplicial surface by a small perturbation of its vertices. Both the change in the angle deficit of a vertex and in the algebraic area of the Gauss image of a vertex star will be small as well. Theorem~\ref{th:egregium} for general polyhedral surfaces then follows by a limiting argument. The shape analysis in Section~\ref{sec:shape_analysis} though requires some care when non-convex faces are present.

Let ${{\bf{v}}}$ denote always a vertex of $P$. Two vertices ${{\bf{v}}}_1$ and ${{\bf{v}}}_2$ are \textit{adjacent} if they are connected by an edge, two faces $f_1$ and $f_2$ are \textit{adjacent} if they share an edge. A face $f$ or an edge $e$ is \textit{incident} to a vertex ${{\bf{v}}}$, if the vertex is a corner of the face or the edge, respectively. For adjacency and incidence we will use the notations ${{\bf{v}}}_1 \sim {{\bf{v}}}_2$, $f_1 \sim f_2$ and $f,e \sim {{\bf{v}}}$, respectively. Our main interest lies in the geometry of the \textit{star} of a vertex ${{\bf{v}}}$, i.e., the set of all faces of $P$ incident to ${{\bf{v}}}$. ${{\bf{v}}}$ is said to be \textit{convex}, if the star of ${{\bf{v}}}$ determines the boundary of a convex polyhedral cone.

The sets of vertices, edges, and faces of $P$ are denoted by $V(P)$, $E(P)$, and $F(P)$, respectively.

We fix an orientation of $P$ and denote by ${\bf{n}}_f \in S^2$ the outer unit normal vector of a face $f$. For a vertex ${{\bf{v}}}$, we connect the normals of adjacent faces ${\bf{n}}_{k-1},{\bf{n}}_{k} \sim {{\bf{v}}}$ by the shorter of the two great circle arcs connecting ${\bf{n}}_{k-1}$ and ${\bf{n}}_{k}$. Since $P$ is embedded, ${\bf{n}}_{k-1}\neq \pm{\bf{n}}_{k}$, so the shorter arc is non-trivial and well defined. Hence, the normals ${\bf{n}}_f$ for $f \sim {{\bf{v}}}$ define a spherical polygon $g({{\bf{v}}})$ that we call the \textit{Gauss image of the vertex star}, see for example Figure~\ref{fig:Gaussian}. The Gauss image inherits an orientation from the orientation of the vertex star. In the example shown in the figure, the orientation is reversed.


\begin{figure}[!ht]
	\centerline{
		\begin{overpic}[height=0.3\textwidth]{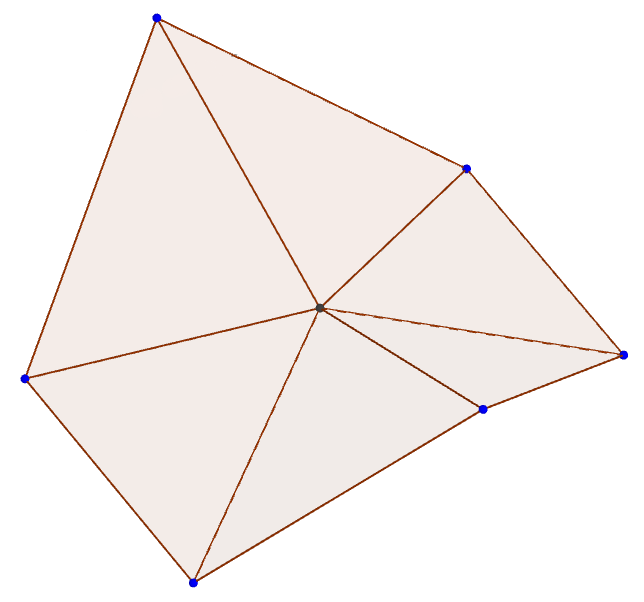}
			\put(32,33){\contour{white}{$f_1$}}
			\put(49,32){\contour{white}{$f_2$}}
			\put(68,36){\contour{white}{$f_3$}}
			\put(70,49){\contour{white}{$f_4$}}
			\put(49,55){\contour{white}{$f_5$}}
			\put(33,50){\contour{white}{$f_6$}}
			\cput(50,47){{$\vw$}}
		\end{overpic}
		\relax\\
		\begin{overpic}[height=0.3\textwidth]{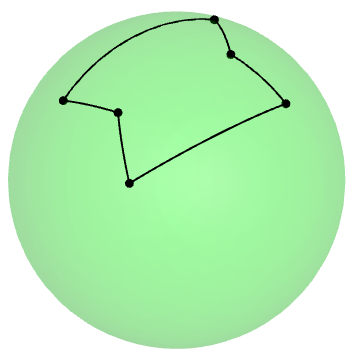}
			\cput(38,69){{$\nw_{1}$}}
			\cput(15,75){{$\nw_{2}$}}
			\cput(59,89){{$\nw_{3}$}}
			\cput(64,78){{$\nw_{4}$}}
			\cput(80,64){{$\nw_{5}$}}
			\cput(40,43){{$\nw_{6}$}}	
		\end{overpic}
	}
	\caption{A vertex star and its Gauss image for a saddle vertex}\label{fig:Gaussian}
\end{figure}
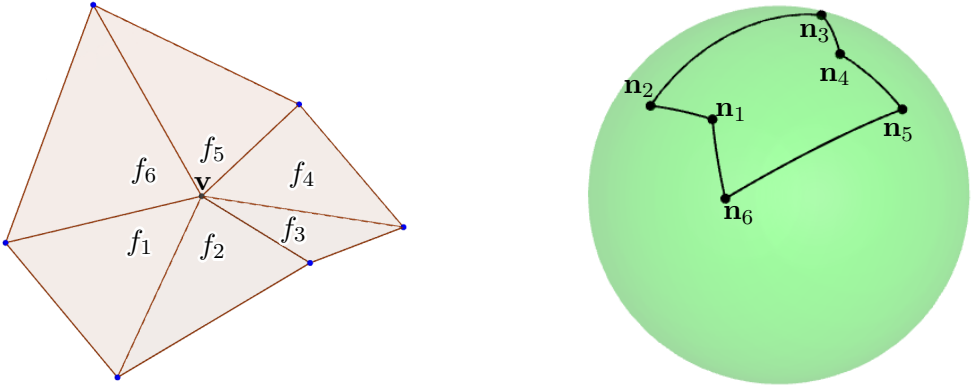

\begin{definition}
Let ${{\bf{v}}}$ be fixed and let $\alpha_f({{\bf{v}}})$ be the interior angle of a face $f\sim{{\bf{v}}}$ at ${{\bf{v}}}$. Then, the \textit{discrete Gaussian curvature} at ${{\bf{v}}}$ is defined as the angle deficit \[K({{\bf{v}}}):=2\pi-\sum\limits_{f \sim {{{\bf{v}}}}}\alpha_f({{\bf{v}}}).\]
\end{definition}

One reason to call the angle deficit discrete Gaussian curvature is the analogy to differential geometry, where Gaussian curvature is the limiting difference between the circumference of a geodesic circle (the equivalent of the angle sum) and a circle in the plane (the equivalent of $2\pi$). A second reason is the validity of the well-known \textit{discrete Gauss-Bonnet theorem}.

\begin{proposition}\label{prop:Gauss_Bonnet}
Let $P$ be a closed simplicial surface. Then, the total discrete Gaussian curvature equals $2\pi$ times the Euler characteristic $\chi(P)$: \[\sum\limits_{{{\bf{v}}}\in V(P)} K({{\bf{v}}})=2\pi\chi(P).\]
\end{proposition}
\begin{proof}
Let $V,E,F$ denote the number of vertices, edges, and triangular faces of $P$, respectively. Then,
\begin{align*}
\sum\limits_{{{\bf{v}}}\in V(P)} K({{\bf{v}}})&=\sum\limits_{{{\bf{v}}}\in V(P)} \left(2\pi-\sum\limits_{f \sim {{{\bf{v}}}}}\alpha_f({{\bf{v}}})\right)\\
&=2\pi V - \sum\limits_{{{\bf{v}}}\in V(P)}\sum\limits_{f \sim {{{\bf{v}}}}}\alpha_f({{\bf{v}}})=2\pi V - \sum\limits_{f\in F(P)}\sum\limits_{{{{\bf{v}}}}\sim f}\alpha_f({{\bf{v}}})\\
&=2\pi V - \pi F = 2\pi V -2\pi E + 2\pi F+ \pi (2E-3F)\\
&=2\pi(V-E+F)=2\pi\chi(P).
\end{align*}
To go from the second to the third line, we have used that the interior angle sum of a triangle is $\pi$. In the last step, we have used $2E=3F$ that follows from the facts that any edge belongs to two faces and any face has three edges.
\end{proof}

In differential geometry, the Gaussian curvature at a point of a smooth surface is also given by the limiting quotient of the areas of the Gaussian images of patches around the point and their areas on the surface. Even though we may intuitively presume how the Gauss image should be extended across a vertex in a way that reflects the geometry of the vertex star, giving a precise definition is not that trivial. Still, the algebraic area of $g({{\bf{v}}})$ can be defined in a way that agrees with our intuition. It equals the discrete Gaussian curvature $K({{\bf{v}}})$. The discussion of the algebraic area and the proof of our main theorem are the content of the following subsections, where we proceed step by step, starting with the simple convex case in Section~\ref{sec:Alexandrov} and ending with the general case in Section~\ref{sec:winding}.

\begin{theorem}\label{th:egregium}
The discrete Gaussian curvature $K({{\bf{v}}})$ equals the algebraic area of the Gauss image $g({{\bf{v}}})$ of the star of the vertex ${{\bf{v}}}$ of a polyhedral surface $P$.
\end{theorem}

Whereas the algebraic area is an extrinsically defined property, the angle deficit depends only on intrinsic measurements made on the polyhedral surface $P$. So Theorem~\ref{th:egregium} is a discretization of the famous \textit{Theorema egregium} of Gauss.

One of the major difficulties in the proof of the theorem are raised by the presence of \textit{inflection faces}. Indeed, if no inflection faces are present, then Section~\ref{sec:Alexandrov} shows how the proof is easily deduced.

\begin{definition}
Let $f$ be a face of the star of the vertex ${{\bf{v}}}$. If the two other faces of the vertex star that are adjacent to $f$ lie in different half-spaces of the plane through $f$, then $f$ is said to be an \textit{inflection face}.
\end{definition}

\begin{example}
Let us have a closer look at Figure~\ref{fig:Gaussian}. The discrete Gaussian curvature is negative. The orientation of the Gauss image is reversed, so its algebraic area is negative, as we expect in the light of Theorem~\ref{th:egregium}. The vertex star is saddle-shaped and has as many inflection faces as a smooth saddle has inflections, namely four: Faces $f_2, f_3, f_5, f_6$ are inflection faces. Their normals exactly correspond to the corners of the Gauss image which is a spherical pseudo-quadrilateral. In Section~\ref{sec:shape_analysis} we actually prove that the only simple spherical polygons that can be the Gauss image of a negatively curved vertex are spherical pseudo-quadrilaterals. In the case that we allow for reflex angles, also spherical pseudo-triangles and pseudo-digons can appear as degenerate cases of pseudo-quadrilaterals.
\end{example}


\subsection{Alexandrov's argument for convex vertex stars}\label{sec:Alexandrov}

The original idea of proving Theorem~\ref{th:egregium} in the convex case goes back to Alexandrov \cite{A05}. He compared the spherical angle $\angle {\bf{n}}_{k+1}{\bf{n}}_k{\bf{n}}_{k-1}$ with the angle $\alpha_{k}({{\bf{v}}})$ in the case that the face $f_k$ is not an inflection face. The spherical angle changes by $\pi$ if $f_k$ becomes an inflection face.

\begin{lemma}\label{lem:angle}
Let $f_1, f_2, f_3$ be three consecutive faces of the star of the vertex ${{\bf{v}}}$, ordered in counterclockwise direction around ${{\bf{v}}}$ with respect to the given orientation of $P$. Let $\alpha:=\alpha_{2}({{\bf{v}}})$ be the angle of $f_2$ at ${{\bf{v}}}$ and $\alpha'$ the spherical angle $\angle {\bf{n}}_{3}{\bf{n}}_{2}{\bf{n}}_{1}$.
\begin{enumerate}
\item If $f_2$ is not an inflection face, then $\alpha'=\pi-\alpha$.
\item If $f_2$ is an inflection face, then $\alpha'=2\pi-\alpha$.
\end{enumerate}
\end{lemma}
\begin{proof}

(i) To each of the two edges of $f_2$ incident to ${{\bf{v}}}$ we consider a plane orthogonal to this edge and intersecting it. Since both ${\bf{n}}_{1}$ and ${\bf{n}}_{2}$ are orthogonal to the edge that is shared by $f_1$ and $f_2$, and ${\bf{n}}_{2}$ and ${\bf{n}}_{3}$ are orthogonal to the edge that is shared by $f_2$ and $f_3$, the spherical angle $\angle {\bf{n}}_{3}{\bf{n}}_{2}{\bf{n}}_{1}$ equals the intersection angle of the two planes we just constructed. More precisely, $\alpha'$ corresponds to the angle opposite to $\alpha$ in the quadrilateral determined by ${{\bf{v}}}$, the two edges of $f_2$ incident to ${{\bf{v}}}$, and the intersection of the two planes with $f_2$. This is due to $f_2$ not being an inflection face. The two angles between these edges and the two planes are given by $\pi/2$, so $\alpha'=\pi-\alpha$.

\begin{figure}[!ht]
	\centerline{
		\begin{overpic}[height=0.3\textwidth]{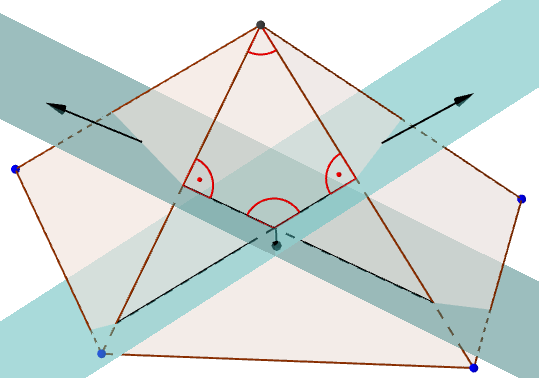}
			\put(18,30){\contour{white}{$f_1$}}
			\put(50,14){\contour{white}{$f_2$}}
			\put(78,32){\contour{white}{$f_3$}}
			\cput(49,67){{$\vw$}}
			\cput(15,44){{$\nw_{1}$}}
			\cput(55,22){{$\nw_{2}$}}
			\cput(89,49){{$\nw_{3}$}}
			\color{red}
			\cput(51,35){{$\alpha'$}}
			\cput(49,57){{$\alpha$}}
		\end{overpic}
		\relax\\
		\begin{overpic}[height=0.3\textwidth]{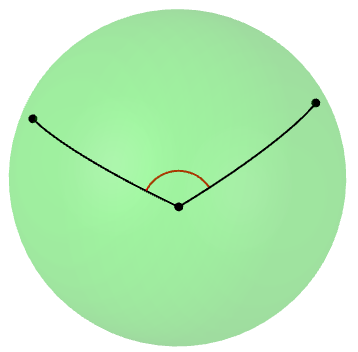}
			\cput(11,58){{$\nw_{1}$}}
			\cput(51,35){{$\nw_{2}$}}
			\cput(89.5,63){{$\nw_{3}$}}
			\color{red}
			\cput(51,44){{$\alpha'$}}
		\end{overpic}
	}
	\caption{Lemma~\ref{lem:angle} (i): $\alpha'=\pi-\alpha$ if $f_2$ is not an inflection face}\label{fig:angle1}
\end{figure}

(ii) Let us consider the edge incident to $f_1$ and $f_2$ as a hinge. If we move $f_1$ to the other side of $f_2$, we obtain the situation of (i) and can apply the result for the spherical angle. So the spherical angle $\angle {\bf{n}}_{3}{\bf{n}}_{2}{\bf{n}}_{1}$ equals $\pi-\alpha$ as long as $f_1$ is on the same side of $f_2$ as $f_3$. If we now are hinging $f_1$ down, ${\bf{n}}_{1}$ moves on the great circle through ${\bf{n}}_{2}$ that intersects the arc ${\bf{n}}_{2}{\bf{n}}_{3}$ in an angle of size $\pi-\alpha$ as in situation (i). When $f_1$ comes to the other side of $f_2$, ${\bf{n}}_{1}$ passes through ${\bf{n}}_{2}$ and the spherical angle $\angle {\bf{n}}_{3}{\bf{n}}_{2}{\bf{n}}_{1}$ becomes larger by $\pi$. Therefore, $\alpha'=\pi+\pi-\alpha=2\pi-\alpha$.

\begin{figure}[!ht]
	\centerline{
		\begin{overpic}[height=0.3\textwidth]{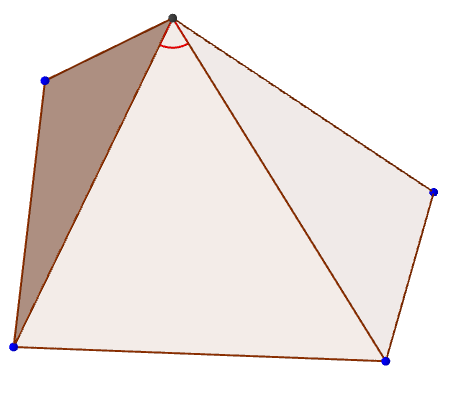}
			\put(13,51){\contour{white}{$f_1$}}
			\put(40,30){\contour{white}{$f_2$}}
			\put(73,40){\contour{white}{$f_3$}}
			\cput(38,85){{$\vw$}}
			\color{red}
			\cput(38,72){{$\alpha$}}
		\end{overpic}
		\relax\\
		\begin{overpic}[height=0.3\textwidth]{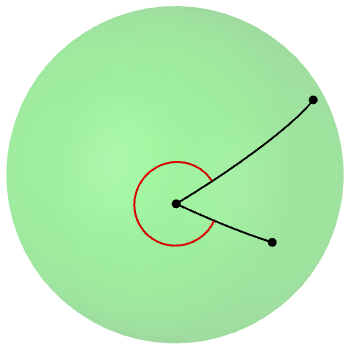}
			\cput(81.5,29){{$\nw_{1}$}}
			\cput(59,39){{$\nw_{2}$}}
			\cput(88,61){{$\nw_{3}$}}
			\color{red}
			\cput(46,40){{$\alpha'$}}
		\end{overpic}
	}
	\caption{Lemma~\ref{lem:angle} (ii): $\alpha'=2\pi-\alpha$ if $f_2$ is not an inflection face}\label{fig:angle2}
\end{figure}
\end{proof}

The area of a spherical $n$-gon equals the sum of its interior angles minus $(n-2)\pi$. By Lemma~\ref{lem:angle}~(i), it follows that the area of the Gauss image of an embedded vertex star without inflection faces equals \[\sum\limits_{f \sim {{\bf{v}}}}\alpha_f'({{\bf{v}}})-(n-2)\pi=\sum\limits_{f \sim {{\bf{v}}}}(\pi-\alpha_f({{\bf{v}}}))-(n-2)\pi=2\pi-\sum\limits_{f \sim {{\bf{v}}}}\alpha_f({{\bf{v}}})=K({{\bf{v}}}).\]

Clearly, the situation changes when ${{\bf{v}}}$ is not a convex corner any longer. We observed in Lemma~\ref{lem:angle}~(ii) that the presence of inflection faces changes the formula for the spherical angle. Moreover, we need more information about the shape of the Gauss image to proceed since the proof for the convex case implicitly used the fact that the Gauss image consists only of positively oriented parts. So already the case of a discrete saddle as in Figure~\ref{fig:Gaussian} is not covered by this argument. Figure~\ref{fig:complicated} shows an even more complicated situation where the Gauss image contains both positively and negatively oriented parts. If we look closer at the figure, we see that the positively oriented part contains normal vectors ${\bf{n}}$ such that the plane orthogonal to ${\bf{n}}$ and passing through ${{\bf{v}}}$ is a tangent plane, i.e., it contains no other points of the vertex star than ${{\bf{v}}}$.

\begin{figure}[!ht]
	\centerline{
		\begin{overpic}[height=0.3\textwidth]{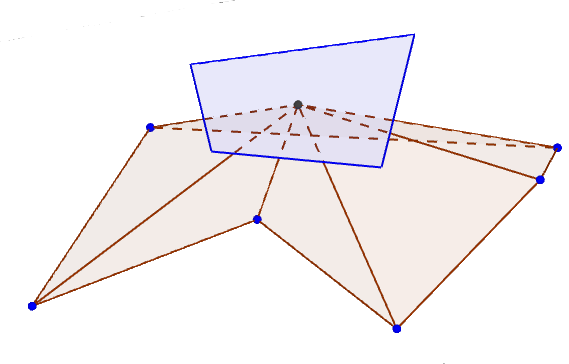}
			\put(70,45){\contour{white}{$f_1$}}
			\put(28,36){\contour{white}{$f_2$}}
			\put(38,30){\contour{white}{$f_3$}}
			\put(50,32){\contour{white}{$f_4$}}
			\put(65,30){\contour{white}{$f_5$}}
			\put(83,37){\contour{white}{$f_6$}}
				\cput(52,47){ ${{\bf{v}}}$}
		\end{overpic}
		\relax\\
		\begin{overpic}[height=0.3\textwidth]{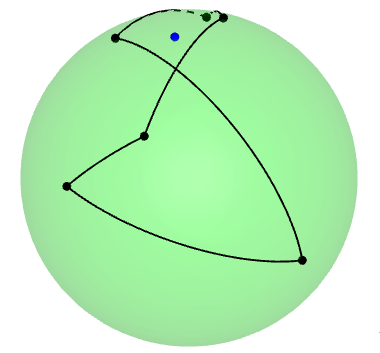}
			\cput(55,91){{$\nw_{1}$}}
 		 \cput(25,83){{$\nw_{2}$}}
			\cput(82,19){{$\nw_{3}$}}
			\cput(14,37){{$\nw_{4}$}}
			\cput(44,55){{$\nw_{5}$}}
			\cput(65,87){{$\nw_{6}$}}	
			\cput(45,85){\color{blue}{$\nw$}}	
		\end{overpic}
	}
	\caption{${\bf{n}}$ in positively oriented part, plane orthogonal to ${\bf{n}}$ intersects the vertex star in 0 line segments}\label{fig:complicated}
\end{figure}
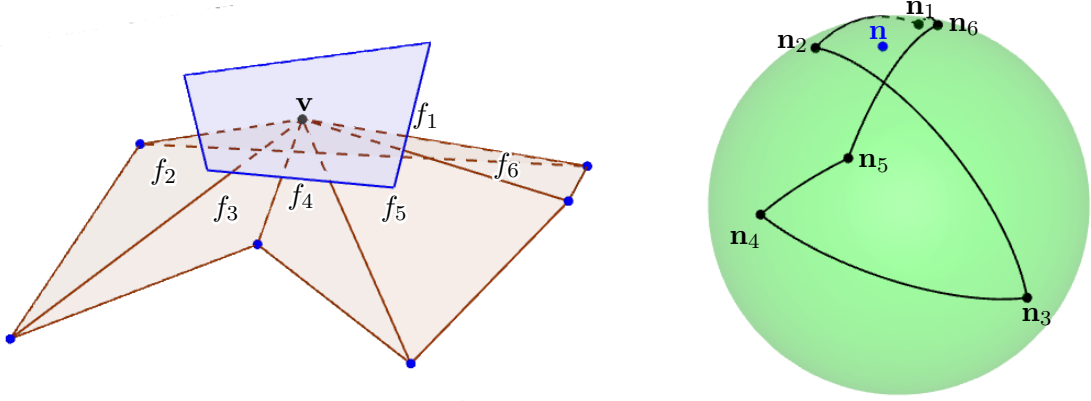

In contrast, the negatively oriented part contains normal vectors ${\bf{n}}$ such that the plane orthogonal to ${\bf{n}}$ and passing through ${{\bf{v}}}$ intersects the vertex star in four line segments. We can see this in Figure~\ref{fig:complicated2}. The other parts of the Gauss image are either antipodal to one of the previously discussed components or the corresponding normal planes intersect the vertex star in only two line segments.

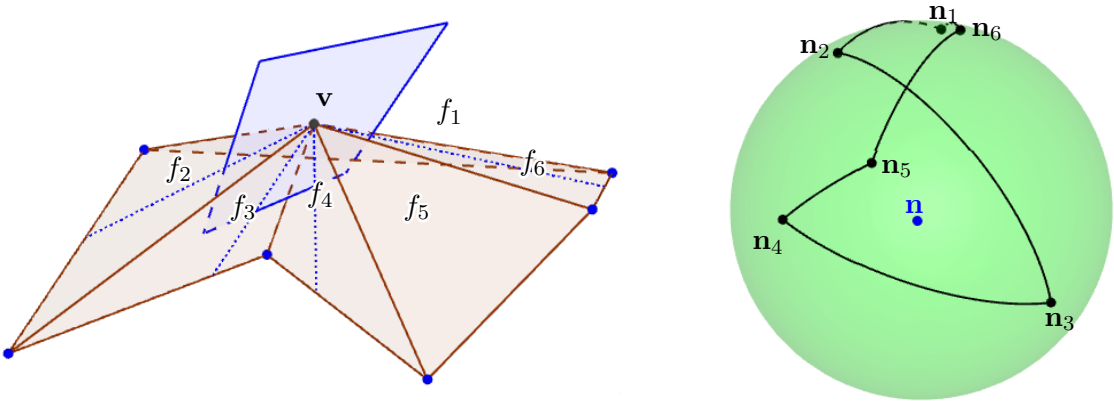
\begin{figure}[!ht]
	\centerline{
		\begin{overpic}[height=0.3\textwidth]{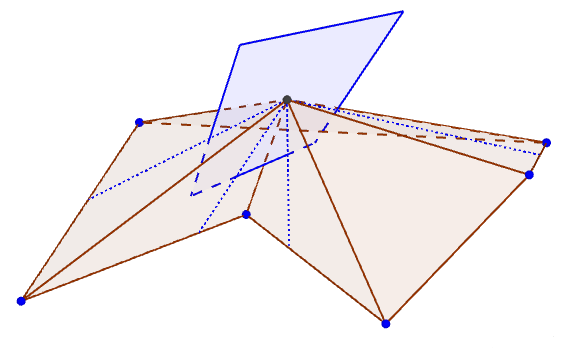}
			\put(70,45){\contour{white}{$f_1$}}
			\put(28,36){\contour{white}{$f_2$}}
			\put(38,30){\contour{white}{$f_3$}}
			\put(50,32){\contour{white}{$f_4$}}
			\put(65,30){\contour{white}{$f_5$}}
			\put(83,37){\contour{white}{$f_6$}}
				\cput(52,47){ ${{\bf{v}}}$}
		\end{overpic}
		\relax\\
		\begin{overpic}[height=0.3\textwidth]{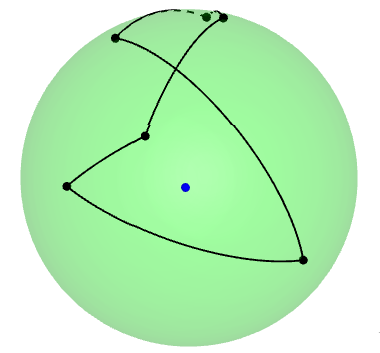}
				\cput(55,91){{$\nw_{1}$}}
 		 \cput(25,83){{$\nw_{2}$}}
			\cput(82,19){{$\nw_{3}$}}
			\cput(14,37){{$\nw_{4}$}}
			\cput(44,55){{$\nw_{5}$}}
			\cput(65,87){{$\nw_{6}$}}	
			\cput(48,45.5){\color{blue}{$\nw$}}	
		\end{overpic}
	}
	\caption{${\bf{n}}$ in negatively oriented part, plane orthogonal to ${\bf{n}}$ intersects the vertex star in 4 line segments}\label{fig:complicated2}
\end{figure}

This observation that the algebraic multiplicity is encoded by the number of line segments in the intersection of a normal plane with the vertex star leads us to the critical point index.


\subsection{Critical point indices of polyhedral surfaces}\label{sec:critical}

The critical point index was defined by the first author as the \textit{above index} in \cite{B67} and as the \textit{middle vertex index} in \cite{B70}. Whereas the above index is better suitable for generalizations in higher dimensions, the middle vertex index is a more intuitive notion for polyhedral surfaces. We will shortly repeat the two notions and show that they are equivalent. Also its basic properties were already discussed in \cite{B67,B70}.

\begin{definition}
Let $\langle \cdot, \cdot \rangle$ denote the Euclidean scalar product. A normal vector $\xi \in S^2$ is called \textit{general} for the vertex ${{\bf{v}}}$, if $\langle \xi, {{\bf{v}}} \rangle \neq \langle \xi, {{\bf{v}}}' \rangle$ for all adjacent vertices ${{\bf{v}}}' \sim {{\bf{v}}}$. If $\xi$ is general for all vertices of the polyhedral surface $P$, then $\xi$ is \textit{general} for $P$.
\end{definition}

\begin{proposition}\label{prop:density}
The set of $\xi \in S^2$ that are general for $P$ is open and dense in $S^2$.
\end{proposition}
\begin{proof}
If $\xi \in S^2$ is not general, then $\langle \xi, {{\bf{v}}} \rangle = \langle \xi, {{\bf{v}}}' \rangle$ for some pair of adjacent vertices ${{\bf{v}}},{{\bf{v}}}'$. This is the case if and only if $\xi$ is on the great circle of normal vectors that are orthogonal to the edge connecting ${{\bf{v}}},{{\bf{v}}}'$. It follows that the set of general normal vectors is the complement of a finite union of great circles on $S^2$. In particular, the set is open and dense in $S^2$.
\end{proof}

\begin{definition}
If $\xi \in S^2$ is general for $P$, let us define for a face $f$ and a vertex ${{\bf{v}}}$ of $P$ the indicator function  \[A(f,{{\bf{v}}},\xi)=\left\{\begin{array}{cl} 1 & \mbox{if }{{\bf{v}}} \sim f \textnormal{ and } \langle \xi, {{\bf{v}}} \rangle \geq \langle \xi, {{\bf{v}}}' \rangle \textnormal{ for all } {{\bf{v}}}' \sim f,\\ 0 & \textnormal{otherwise.} \end{array}\right.\] Similarly, $A(e,{{\bf{v}}},\xi)$ is defined for an edge $e$ of $P$.

The \textit{(above) index of ${{\bf{v}}}$ with respect to $\xi$} is defined as \[i({{\bf{v}}},\xi):=1-\sum\limits_{e \in E(P)} A(e,{{\bf{v}}},\xi)+\sum\limits_{f \in F(P)} A(f,{{\bf{v}}},\xi).\]
\end{definition}

\begin{corollary}\label{cor:continuity}
If $\xi \in S^2$ is general for the star of a vertex ${{\bf{v}}}$, then $i({{\bf{v}}},\xi)=i({{\bf{v}}},\xi')$ for every $\xi'\in S^2$ in a neighborhood of $\xi$.
\end{corollary}

In analogy to classical Morse theory, the critical point theorem holds true \cite{B67,B70}.

\begin{proposition}\label{prop:critical}
If $\xi \in S^2$ is general for $P$, then \[\sum\limits_{{{\bf{v}}}\in V(P)} i({{\bf{v}}},\xi)=\chi(P).\]
\end{proposition}
\begin{proof}
Let $V,E,F$ denote the number of vertices, edges, and faces of $P$, respectively. Since $\langle \xi , \cdot \rangle$ attains its maximum at exactly one vertex per edge or face, we have \[\sum\limits_{{{\bf{v}}}\in V(P)}\sum\limits_{e\in E(P)} A(e, {{\bf{v}}},\xi)=\sum\limits_{e\in E(P)}\sum\limits_{{{\bf{v}}}\in V(P)} A(e, {{\bf{v}}},\xi)=\sum\limits_{e\in E(P)} 1=E \textnormal{ and } \sum\limits_{{{\bf{v}}}\in V(P)}\sum\limits_{f\in F(P)} A(f, {{\bf{v}}},\xi)=F.\] So in the end, we get \[\sum\limits_{{{\bf{v}}}\in V(P)} i({{\bf{v}}},\xi)=\sum\limits_{{{\bf{v}}}\in V(P)} \left(1-\sum\limits_{e\in E(P)}A(e, {{\bf{v}}},\xi)+\sum\limits_{f\in F(P)}A(f, {{\bf{v}}},\xi)\right)=V-E+F=\chi(P).\qedhere\]
\end{proof}

\begin{proposition}\label{prop:index_equality}
The above index of ${{\bf{v}}}$ with respect to a general $\xi \in S^2$ equals the \textit{middle vertex index}: \[i({{\bf{v}}},\xi)=1-\frac{1}{2}M({{\bf{v}}},\xi).\] Here, $M({{\bf{v}}},\xi)$ is the number of faces $f \sim {{\bf{v}}}$ such that  $\langle \xi, \cdot \rangle$ does neither attains its maximum or its minimum on $f$ at ${{\bf{v}}}$, meaning that ${{\bf{v}}}$ is \textit{middle} for $\xi$. 
\end{proposition}
\begin{proof}
We consider the orientation of the star of ${{\bf{v}}}$ inherited by $P$. Let ${{\bf{v}}},{{\bf{v}}}',{{\bf{v}}}''$ be the vertices of a triangle $f$ in counterclockwise order. Let $M_+$ be the set of triangles such that $\langle \xi, {{\bf{v}}}' \rangle  < \langle \xi, {{\bf{v}}} \rangle < \langle \xi, {{\bf{v}}}'' \rangle$, and let $M_-$ be the set of triangles such that $\langle \xi, {{\bf{v}}}' \rangle  > \langle \xi, {{\bf{v}}} \rangle > \langle \xi, {{\bf{v}}}'' \rangle$. Then, $\vert M_+\vert +\vert M_-\vert =M({{\bf{v}}},\xi)$ and $\vert M_+\vert =\vert M_-\vert $.

Let us denote by $e_f$ the edge ${{\bf{v}}}{{\bf{v}}}'$. Then, $A(f,{{\bf{v}}},\xi)=1$ if and only if $A(e_f,{{\bf{v}}},\xi)=1$ and $f \notin M_+$. By construction, $A(e_f,{{\bf{v}}},\xi)=1$ for all $f \in M_+$. Hence,
\begin{align*}
i({{\bf{v}}},\xi)&=1-\sum\limits_{e \in E(P)} A(e,{{\bf{v}}},\xi)+\sum\limits_{f \in F(P)} A(e,{{\bf{v}}},\xi)\\
&=1-\sum\limits_{e \in E(P)} A(e,{{\bf{v}}},\xi)+\sum\limits_{f \in F(P)} A(e_f,{{\bf{v}}},\xi)-\vert M_+\vert \\
&=1-\vert M_+\vert =1-\frac{1}{2}M({{\bf{v}}},\xi).\qedhere
\end{align*}
\end{proof}

Note that ${{\bf{v}}}$ is middle for $\xi$ in a triangle $f$ exactly if the plane orthogonal to $\xi$ and passing through ${{\bf{v}}}$ intersects $f$ in a line segment. Therefore, $i({{\bf{v}}},\xi)=1$ if and only if the normal plane to $\xi$ does not contain any points of the vertex star other than ${{\bf{v}}}$, as it is the case in Figure~\ref{fig:complicated}. If we consider $\langle \xi, \cdot \rangle$ as a height function of $P$, this means that ${{\bf{v}}}$ is a \textit{local maximum} or \textit{local minimum}. If $i({{\bf{v}}},\xi)=0$, one connected part of the vertex star lies lower than ${{\bf{v}}}$ and the complementary part lies higher: ${{\bf{v}}}$ is an \textit{ordinary point}. $i({{\bf{v}}},\xi)=-1$ holds if there are four line segments in the intersection of the vertex star with the plane orthogonal to $\xi$ as in Figure~\ref{fig:complicated2}. Then, ${{\bf{v}}}$ is a \textit{saddle point}.

For a smooth surface in Euclidean space, it is a classical result that for almost every $\xi \in S^2$ the corresponding height function has only finitely many critical points, and that this height function has as critical points only local extrema and nondegenerate saddle points. In the polyhedral case, however, saddle points of index $i({{\bf{v}}},\xi)\leq-2$ may be stable, as is the monkey saddle depicted in Figure~\ref{fig:monkey}. But as before, we see that $i({{\bf{v}}},\xi)$ equals the algebraic multiplicity of $\xi$ or $-\xi$ in the Gauss image $g({{\bf{v}}})$, noting that the index takes the same value for antipodal points.

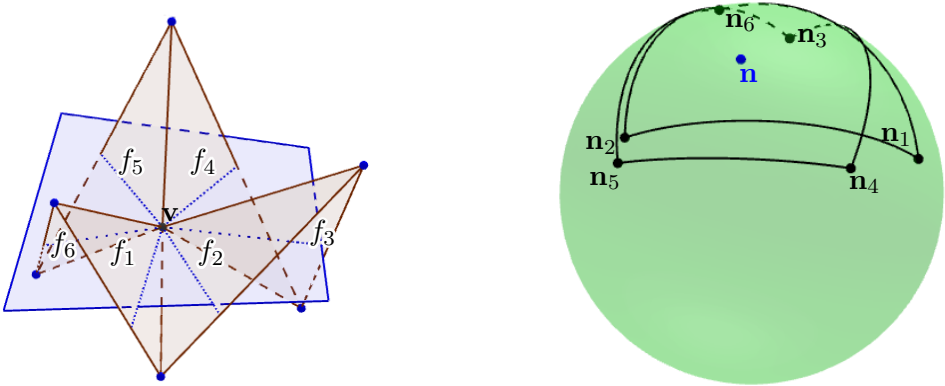
\begin{figure}[!ht]
	\centerline{
		\begin{overpic}[height=0.3\textwidth]{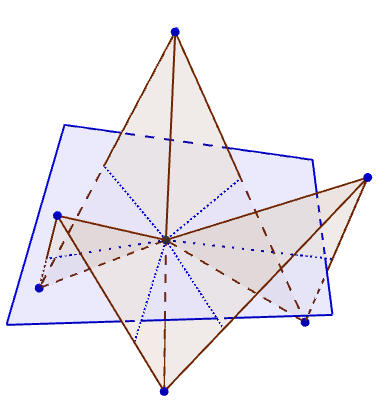}
			\put(28,32){\contour{white}{$f_1$}}
			\put(50,32){\contour{white}{$f_2$}}
			\put(78,37){\contour{white}{$f_3$}}
			\put(48,55){\contour{white}{$f_4$}}
			\put(30,55){\contour{white}{$f_5$}}
			\put(13,34){\contour{white}{$f_6$}}
				\cput(42,42){ ${{\bf{v}}}$}
		\end{overpic}
		\relax\\
		\begin{overpic}[height=0.3\textwidth]{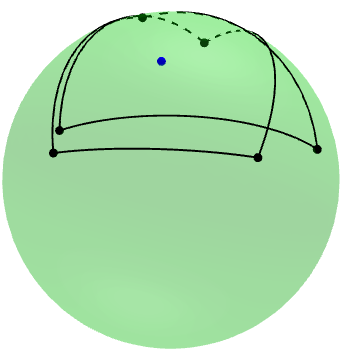}
				\cput(85,62){{$\nw_{1}$}}
 		 \cput(11,60){{$\nw_{2}$}}
			\cput(64,87){{$\nw_{3}$}}
			\cput(77,50){{$\nw_{4}$}}
			\cput(12,51){{$\nw_{5}$}}
			\cput(46,91){{$\nw_{6}$}}	
			\cput(48,77){\color{blue}{$\nw$}}	
		\end{overpic}
	}
	\caption{Index $i({{\bf{v}}},{\bf{n}})=-2$ in the case of a monkey saddle}\label{fig:monkey}
\end{figure}

Therefore, we expect that integration of $i({{\bf{v}}},\xi)$ with respect to $\xi \in S^2$ yields twice the algebraic area of $g({{\bf{v}}})$. In the light of Theorem~\ref{th:egregium}, this should equal twice the discrete Gaussian curvature, which is indeed the case \cite{B67,B70}.

\begin{proposition}\label{prop:integration_index}
Let ${{\bf{v}}}$ be a vertex of a polyhedral surface. Then, \[K({{\bf{v}}})=\frac{1}{2}\int\limits_{S^2} i({{\bf{v}}},\xi)d\xi.\]
\end{proposition}
\begin{proof}
We integrate each term of \[i({{\bf{v}}},\xi)=1-\sum\limits_{e \in E(P)} A(e,{{\bf{v}}},\xi)+\sum\limits_{f \in F(P)} A(f,{{\bf{v}}},\xi)\] separately. Integration of $1$ over $S^2$ gives the area of the sphere, which is $4\pi$.

Let us orient the edge $e={{\bf{v}}}{{\bf{v}}}'$ from ${{\bf{v}}}$ to ${{\bf{v}}}'$. Then, $\langle \xi, {{\bf{v}}} \rangle \geq \langle \xi, {{\bf{v}}}' \rangle$ if and only if $\langle \xi, e \rangle \leq 0$. This equation defines a hemisphere of normal vectors $\xi$ for which $A(e,{{\bf{v}}},\xi)=1$. It follows that \[\int\limits_{S^2}\sum\limits_{e \in E(P)} A(e,{{\bf{v}}},\xi)=2\pi n,\] where $n$ is the number of edges incident to ${{\bf{v}}}$.

Let us now examine when $A(f,{{\bf{v}}},\xi)=1$ for a triangular face $f$ with vertices ${{\bf{v}}},{{\bf{v}}}',{{\bf{v}}}''$. If we decompose $\xi={\bf{n}}+\eta$ into a vector ${\bf{n}}$ orthogonal to $f$ and a vector $\eta$ parallel to $f$, then only $\eta$ determines whether $A(f,{{\bf{v}}},\xi)=1$ or not. More precisely, $A(f,{{\bf{v}}},\xi)=1$ if and only if $\eta$ is contained in the conical sector perpendicular to the sector determined by the interior angle $\alpha=\alpha_f({{\bf{v}}})$, see Figure~\ref{fig:index_angle}. The angle of the conical sector hence equals $\pi-\alpha$. It follows that the portion of vectors $\eta$ and thus of vectors $\xi$ that lead to $A(f,{{\bf{v}}},\xi)=1$ equals $(\pi-\alpha)/2\pi$. So  \[\int\limits_{S^2}\sum\limits_{f \in F(P)} A(f,{{\bf{v}}},\xi)=\sum\limits_{f \in F(P)} 2(\pi-\alpha_f({{\bf{v}}}))=2\pi n -2\sum\limits_{f \in F(P)} \alpha_f({{\bf{v}}}).\]

\begin{figure}[!ht]
	\centerline{
		\begin{overpic}[height=0.2\textwidth]{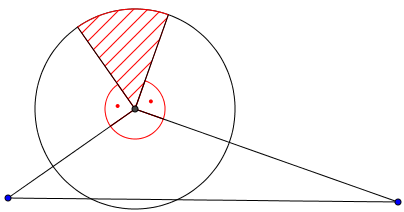}
			\cput(33,28){${{\bf{v}}}$}
			\cput(0,3){${{\bf{v}}}'$}
			\cput(101,3){${{\bf{v}}}''$}
			\color{red}
			\cput(34,20){{$\alpha$}}
		\end{overpic}
	}
	\caption{$A(f,{{\bf{v}}},\xi)=1$ if and only if the component $\eta$ of $\xi$ parallel to $f$ is in the shaded region}\label{fig:index_angle}
\end{figure}

In summary, \[\frac{1}{2}\int\limits_{S^2} i({{\bf{v}}},\xi)d\xi=\frac{1}{2}\left(4\pi-2\pi n +2\pi n -2\sum\limits_{f \in F(P)} \alpha_f({{\bf{v}}})\right)=2\pi-\sum\limits_{f \in F(P)} \alpha_f({{\bf{v}}})=K({{\bf{v}}}).\qedhere\]
\end{proof}

Unfortunately, Proposition~\ref{prop:integration_index} does not yet give the full proof of Theorem~\ref{th:egregium}. We still have to show the correspondence between the index and algebraic multiplicity. Furthermore, $i({{\bf{v}}},\xi)$ does in general not agree with the algebraic multiplicity of $\xi$ or $-\xi$ in the Gauss image $g({{\bf{v}}})$. Such a disagreement occurs if the Gauss image contains pairs of antipodal points, as it does for example in the case of Cs\'asz\'ar's torus illustrated in Figure~\ref{fig:csaszar}. Cs\'asz\'ar's torus is the only known embedded polyhedron besides the tetrahedron without interior diagonals \cite{C49}. In Figure~\ref{fig:csaszar}, we see that both ${\bf{n}}$ and $-{\bf{n}}$ are contained in the Gauss image. Furthermore, the plane orthogonal to ${\bf{n}}$ intersects the vertex star in six line segments such that $i({{\bf{v}}},{\bf{n}})=-2$ as in the case of the monkey saddle in Figure~\ref{fig:monkey}. However, the algebraic multiplicity of ${\bf{n}}$ in $g({{\bf{v}}})$ is -1. The same is true for $-{\bf{n}}$; before, the algebraic multiplicity of the antipode was 0. In any case, the sum of the algebraic multiplicities of ${\bf{n}}$ and $-{\bf{n}}$ equals the index. We prove this claim in the following Section~\ref{sec:winding} and by that complete our proof of Theorem~\ref{th:egregium}.

\begin{figure}[!ht]
	\centerline{
		\begin{overpic}[height=0.3\textwidth]{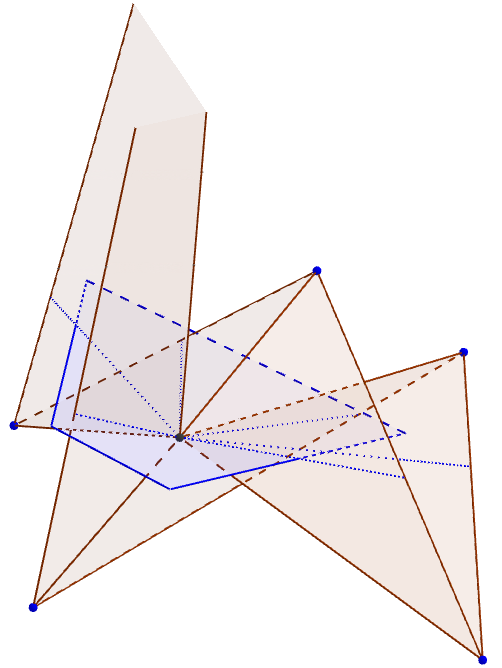}
			\put(5,55){\contour{white}{$f_1$}}
			\put(13,34){\contour{white}{$f_2$}}
			\put(18,20){\contour{white}{$f_3$}}
			\put(50,20){\contour{white}{$f_4$}}
			\put(61,30){\contour{white}{$f_5$}}
			\put(32,52){\contour{white}{$f_6$}}
				\cput(28,37){ ${{\bf{v}}}$}
		\end{overpic}
		\relax\\
		\begin{overpic}[height=0.3\textwidth]{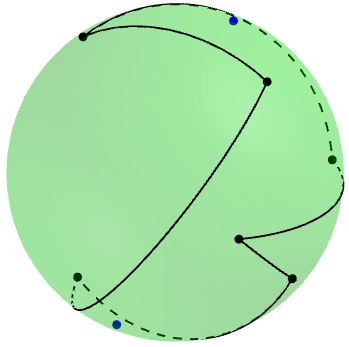}
			\cput(89,53){{$\nw_{1}$}}
 		 \cput(64,26){{$\nw_{2}$}}
			\cput(78,18){{$\nw_{3}$}}
			\cput(17,17){{$\nw_{4}$}}
			\cput(70,74){{$\nw_{5}$}}
			\cput(26,83){{$\nw_{6}$}}	
			\cput(65,88){\color{blue}{$\nw$}}
			\cput(26,4){\color{blue}{$-\nw$}}	
		\end{overpic}
	}
	\caption{Index $i({{\bf{v}}},{\bf{n}})=-2$ in the case of Cs\'asz\'ar's torus}\label{fig:csaszar}
\end{figure}
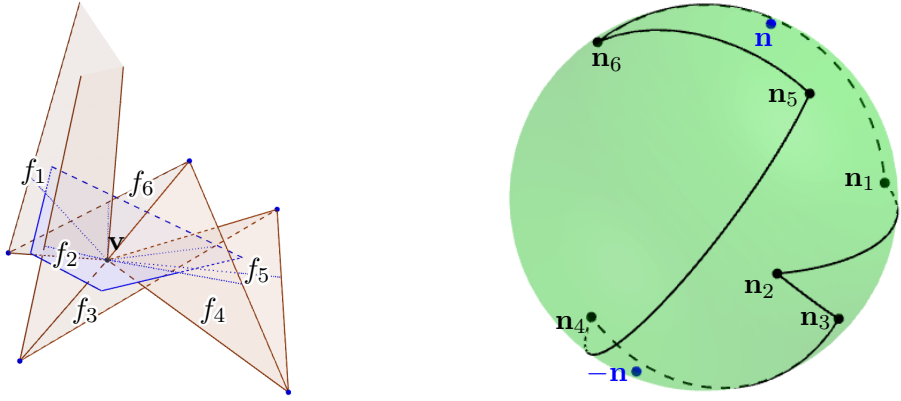

Before that, we introduce a class of polyhedral vertex stars for which index and algebraic multiplicity coincide as long as the normal vector is contained in the Gauss image. This is the case for vertex stars possessing a transverse plane.

\begin{definition}\label{def:transverse}
A plane $E$ passing through the vertex ${{\bf{v}}}$ of $P$ is said to be a \textit{transverse plane} if the vertex star projects orthogonally to $E$ in a one-to-one way.
\end{definition}

This definition mimics the corresponding property of the tangent plane of a smooth surface and is part of the definition of a smooth polyhedral surface in \cite{GJP16}. Existence of a transverse plane is equivalent to the condition that $g({{\bf{v}}})$ is contained in an open hemisphere. This implies that $g({{\bf{v}}})$ does not contain any pair of antipodal points.

\begin{proposition}\label{prop:transverse}
There exists a transverse plane through ${{\bf{v}}}$ if and only if $g({{\bf{v}}})$ is contained in an open hemisphere.
\end{proposition}
\begin{proof}
Let $E$ be a transverse plane through ${{\bf{v}}}$, and let ${\bf{n}}$ be the normal of $E$ that is consistent with the orientation of $P$, that means that the counterclockwise direction around the star of ${{\bf{v}}}$ agrees with the counterclockwise direction around ${\bf{n}}$.

\begin{figure}[!ht]
	\centerline{
		\begin{overpic}[height=0.3\textwidth]{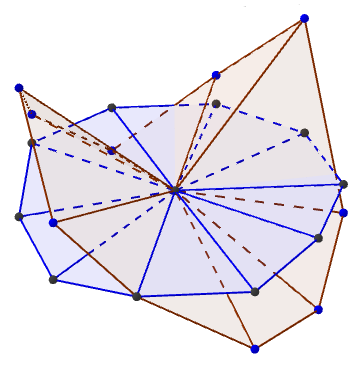}
	\cput(47,51){{$\vw$}}	
		\end{overpic}
		\relax\\
		\begin{overpic}[height=0.3\textwidth]{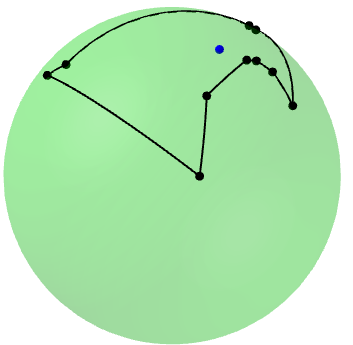}
			\cput(58,80){\color{blue}{$\nw$}}	
		\end{overpic}
	}
	\caption{Relation between the existence of a transverse plane and the Gauss image lying in a hemisphere}\label{fig:transversal}
\end{figure}

Let us investigate the orientation of the image of a face $f \sim {{\bf{v}}}$ under orthogonal projection to $E$ (when seen from the side of ${\bf{n}}$). It coincides with the orientation of $f$ if ${\bf{n}}_f$ is contained in the upper hemisphere with pole ${\bf{n}}$; it degenerates to a line segment if ${\bf{n}}_f$ is contained in the equator; and the orientation is reversed if ${\bf{n}}_f$ is contained in the lower hemisphere. Since the star of ${{\bf{v}}}$ projects to $E$ in a one-to-one way, it follows that for all faces $f \sim {{\bf{v}}}$, ${\bf{n}}_f$ is contained in the upper hemisphere with pole ${\bf{n}}$. In particular, $g({{\bf{v}}})$ is contained in this open hemisphere.

Conversely, if $g({{\bf{v}}})$ is contained in the open hemisphere defined by a vector ${\bf{n}} \in S^2$, the same argument shows that the star of ${{\bf{v}}}$ orthogonally projects to the plane $E$ through ${{\bf{v}}}$ orthogonal to ${\bf{n}}$ in a bijective and orientation-preserving way.
\end{proof}


\subsection{Comparison of critical point indices with winding numbers}\label{sec:winding}

It is now time to define the algebraic multiplicity of a normal vector in the Gauss image. Mimicking the planar case, it should be given by the winding number of the Gauss image around that normal vector. However, we cannot a priori distinguish the inside and the outside of the Gauss image, such that just taking the winding number may not give the correct result. For example, the winding number of a point ``outside'' a Gauss image without self-intersections is $\pm 1$ depending on the orientation, but the algebraic multiplicity is 0. If we just see the point and the oriented Gauss image, it is even hard to tell what the algebraic multiplicity should be. On the sphere, we still get a meaningful result if we add a common integer to the algebraic multiplicities of normal vectors.

We circumvent this problem by defining our winding number using a description of the behavior when an arc of the Gauss image is crossed. To obtain the correct normalization, we use continuity and a convex vertex star as a particular reference case where the algebraic multiplicity can be clearly defined.

\begin{definition}
Let $P$ be an embedded polyhedral surface and let ${{\bf{v}}}$ be a vertex. The \textit{winding number} $w(g({{\bf{v}}}),\xi)=w_{{{\bf{v}}},P}(g({{\bf{v}}}),\xi)$ of the oriented Gauss image $g({{\bf{v}}})$ around a normal vector $\xi \in S^2$ that does not lie on $g({{\bf{v}}})$ is the value of the unique function satisfying the following properties:
\begin{enumerate}
\item In any connected component in the complement of $g({{\bf{v}}})$, $w(g({{\bf{v}}}),\cdot)$ is constant.
\item If $\xi$ crosses $g({{\bf{v}}})$ from left to right, the winding number decreases by 1.
\item Under continuous deformations of $P$, the winding number $w(g({{\bf{v}}}),\xi)$ does not change as long $\xi$ never lies on $g({{\bf{v}}})$.
\item If ${{\bf{v}}}$ is a convex corner, $w(g({{\bf{v}}}),\xi)=1$ if $\xi$ is contained in the component of the complement of $g({{\bf{v}}})$ of area less than $2\pi$, and $w(g({{\bf{v}}}),\xi)=0$ if $\xi$ is contained in the component of the complement of $g({{\bf{v}}})$ of area larger than $2\pi$.
\end{enumerate}
\end{definition}

\begin{figure}[!ht]
	\centerline{
		\begin{overpic}[height=0.3\textwidth]{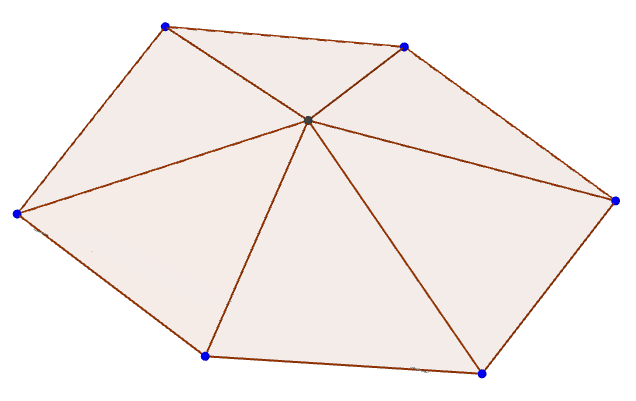}
	\cput(49,46){{$\vw$}}	
		\end{overpic}
		\relax\\
		\begin{overpic}[height=0.3\textwidth]{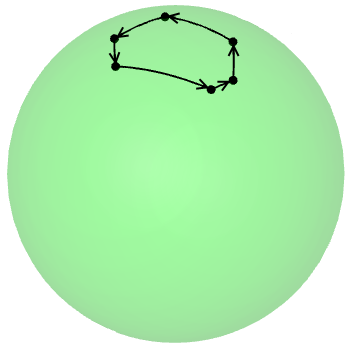}
			\cput(47,82){{$1$}}
			\cput(47,41){{$0$}}	
		\end{overpic}
	}
	\caption{The winding number is constant in components of the complement, it decreases by 1 if one crosses the Gauss image from left to right, and it is normalized by the values depicted in the figure}\label{fig:convex}
\end{figure}

\begin{remark}
We can center a ball at ${{\bf{v}}}$ and project the vertex star onto it. The result is a simple spherical polygon. Since the vertex star contains no reflex angles, the length of any edge of the polygon is less than $\pi$. Conversely, we can construct an embedded vertex star from a given simple spherical polygon with edge lengths smaller than $\pi$ by connecting the polygon to the center of the ball. Since the polygon is simple, we can define its interior and its exterior on the sphere. Now, deformations of simple polygons can be described by a movement of an interior point and the change of the polygon with respect to that point.

Suppose we have a deformation going from a given polygon to itself. Since the sphere is simply-connected, the path of the chosen interior point is null-homotopic. By stereographic projection, the change of the spherical polygon with respect to the interior point is described by a deformation of a simple (piecewise smooth) curve on the plane. The space of the latter is simply-connected, and so is the space of embedded vertex stars (of a given combinatoric). In particular, the winding number is well-defined.
\end{remark}

\begin{example}
Figure~\ref{fig:convex} illustrates the winding numbers in the convex case. In the case of the saddle in Figure~\ref{fig:Gaussian}, the winding number is -1 if $\xi$ is contained in the smaller of the two components defined by the Gauss image, and $0$ in the larger component. If we consider the polyhedral vertex star depicted in Figures~\ref{fig:complicated} and~\ref{fig:complicated2}, $w(g({{\bf{v}}}),\xi)=1$ if $\xi={\bf{n}}$ as in Figure~\ref{fig:complicated}, $w(g({{\bf{v}}}),\xi)=-1$ if $\xi={\bf{n}}$ as in Figure~\ref{fig:complicated2}, and $w(g({{\bf{v}}}),\xi)=0$ in the remaining component. In the case of the monkey saddle in Figure~\ref{fig:monkey}, $w(g({{\bf{v}}}),\xi)=-2$ and $w(g({{\bf{v}}}),-\xi)=0$ if $\xi={\bf{n}}$ as shown in that figure. In the remaining components, $w(g({{\bf{v}}}),\xi)=-1$. Finally,  $w(g({{\bf{v}}}),\xi)=w(g({{\bf{v}}}),-\xi)=-1$ if $\xi={\bf{n}}$ as shown in Figure~\ref{fig:csaszar} for Cs\'asz\'ar's torus -- in the other component of the complement of $g({{\bf{v}}})$, the winding number is zero.
\end{example}

As already announced in the previous section, $i({{\bf{v}}},\xi)$ equals the sum of the winding numbers of $g({{\bf{v}}})$ around $\xi$ and $-\xi$.

\begin{proposition}\label{prop:winding}
Let $\xi\in S^2$ be general for the vertex ${{\bf{v}}}$ of an embedded polyhedral surface. Then, \begin{align}i({{\bf{v}}},\xi)=w(g({{\bf{v}}}),\xi)+w(g({{\bf{v}}}),-\xi). \label{eq:winding}\end{align}
\end{proposition}
\begin{proof}
To each edge $e$ incident to ${{\bf{v}}}$ we associate the great circle $c_e$ that is given by the intersection of the plane orthogonal to $e$ and passing through the origin with $S^2$. The set of circles $c_e$ yields a decomposition of the sphere into convex spherical cells. We have to show Equation~(\ref{eq:winding}) for each of these cells. It follows from the definition and Corollary~\ref{cor:continuity} that both the winding number and the index are constant on each cell. Furthermore, they both do not change under continuous deformations of $P$ as long as $\xi$ stays general for ${{\bf{v}}}$ throughout the deformation. So without loss of generality, we assume that great circles associated to different edges are different. In particular, $g({{\bf{v}}})$ and $-g({{\bf{v}}})$ intersect only transversely.

Let us first consider the case of a convex corner as shown in Figure~\ref{fig:convex}. By Proposition~\ref{prop:index_equality}, where we have shown equality of the above index and the middle vertex index, $i({{\bf{v}}},\xi)=1$ if and only if the plane orthogonal to $\xi$ and passing through ${{\bf{v}}}$ does not contain any further points of the vertex star. Each face defines a double lune of normal vectors $\xi$ such that the corresponding plane does not intersect the face. It is easy to see that the intersection of all these double lunes gives exactly the component enclosed counterclockwise by $g({{\bf{v}}})$ together with its antipode. In the remaining part, $i({{\bf{v}}},\xi)=0$. It follows that $i({{\bf{v}}},\xi)=w(g({{\bf{v}}}),\xi)+w(g({{\bf{v}}}),-\xi)$ in this case.

Recalling the definition of the winding number, in order to prove Equation~(\ref{eq:winding}) it suffices to show that $i({{\bf{v}}},\xi)$ and $w(g({{\bf{v}}}),\xi)+w(g({{\bf{v}}}),-\xi)$ behave the same when a great circle $c_e$ is crossed. That means that $i({{\bf{v}}},\xi)$ should not change when we cross a part of $c_e$ that is neither contained in $g({{\bf{v}}})$ or $-g({{\bf{v}}})$; it should decrease by $1$ if we cross a part of $c_e$ that is contained in $g({{\bf{v}}})$ from left to right according to the orientation of the Gauss image; and it should increase by $1$ if we cross a part of $c_e$ that is contained in $-g({{\bf{v}}})$ from left to right, where $-g({{\bf{v}}})$ inherits the orientation of $g({{\bf{v}}})$. Note that $-g({{\bf{v}}})$ is negatively oriented if $g({{\bf{v}}})$ is oriented positively.

All the information that we need to prove this claim are just the positions of the vertices of the two faces incident to the edge $e$. Let us denote the vertices other than ${{\bf{v}}}$ by ${{\bf{v}}}_1,{{\bf{v}}}_2,{{\bf{v}}}_3$ in counterclockwise order, and let $e_i$ be the edge connecting ${{\bf{v}}}$ with ${{\bf{v}}}_i$. In particular, $e=e_2$. We denote by $f_1$ the face with vertices ${{\bf{v}}},{{\bf{v}}}_1,{{\bf{v}}}_2$ and by $f_2$ the face with vertices ${{\bf{v}}},{{\bf{v}}}_2,{{\bf{v}}}_3$. ${\bf{n}}_1$ and ${\bf{n}}_2$ shall be its normal vectors. 

Since the index inside the cells defined by all the $c_{e'}$, $e'\sim{{\bf{v}}}$, does not change as long we do not change the structure of the cell decomposition, we can deform the vertices of the star without changing the indices as long as no two faces become coplanar. Since we are only interested in the positions of ${{\bf{v}}},{{\bf{v}}}_1,{{\bf{v}}}_2,{{\bf{v}}}_3$, we may hence assume that ${{\bf{v}}}$ lies at the origin of $\mathds{R}^3$, that ${{\bf{v}}}_2$ has coordinates $(0,1,0)$ and that ${{\bf{v}}}_1$ and ${{\bf{v}}}_3$ have coordinates $(1,0,0)$ and $(0,0,1)$. The choice whether ${{\bf{v}}}_1$ or ${{\bf{v}}}_3$ has coordinates $(1,0,0)$ determines the orientation of the normal vectors ${\bf{n}}_j$. We start with the case that ${{\bf{v}}}_1$ has coordinates $(1,0,0)$, as illustrated in Figure~\ref{fig:cube}. Then, ${\bf{n}}_1$ has coordinates $(0,0,1)$ and ${\bf{n}}_2$ has coordinates $(1,0,0)$. The oriented great circle arc $a$ from ${\bf{n}}_1$ to ${\bf{n}}_2$ is part of the Gauss image $g({{\bf{v}}})$.

\begin{figure}[!ht]
	\centerline{
		\begin{overpic}[height=0.3\textwidth]{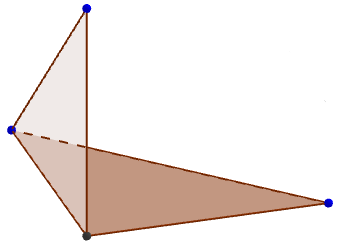}
			\put(49,15){\contour{white}{$f_1$}}
			\put(17,47){\contour{white}{$f_2$}}
			\cput(22,0){{$\vw$}}
			\cput(65,5){{$e_1$}}
			\cput(29,35){{$e_3$}}
			\cput(10,15){{$e_2$}}
			\color{blue}
			\cput(100,8){{$\vw_1$}}
			\cput(0,31){{$\vw_2$}}
			\cput(26,73){{$\vw_3$}}
		\end{overpic}
		\relax\\
		\begin{overpic}[height=0.3\textwidth]{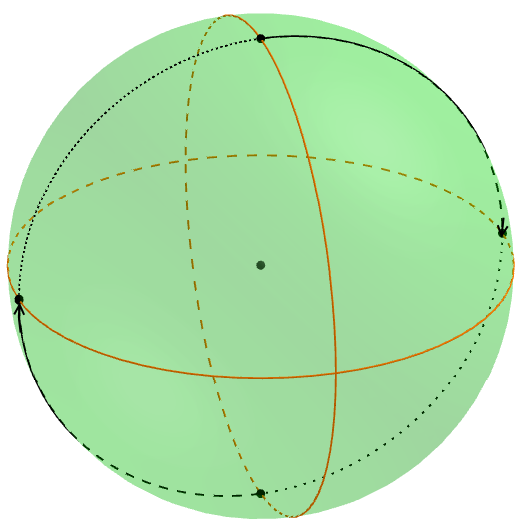}
			\cput(49,95){{$\nw_{1}$}}
			\cput(102,54){{$\nw_{2}$}}
			\cput(49,0){{$-\nw_{1}$}}
			\cput(-5,40){{$-\nw_{2}$}}
			\cput(50,44){{$\vw$}}
			\cput(80,80){{$a$}}
			\cput(6,15){{$-a$}}
			\cput(20,70){{$c_e$}}
			\color{orange}
			\cput(77,67){{$\left\{\xi:\xi_1=0\right\}$}}
			\cput(30,36){{$\left\{\xi:\xi_3=0\right\}$}}	
		\end{overpic}
	}
	\caption{Change of the index $i({{\bf{v}}},\xi)$ when $\xi$ crosses a part of $c_e$}\label{fig:cube}
\end{figure}

Let $\xi=(\xi_1,\xi_2,\xi_3)$. Then, $\xi$ is general if and only if $\xi_1,\xi_2,\xi_3\neq 0$. It is easy to see that $A(e_i,{{\bf{v}}},\xi)=1$ if and only if $\xi_i < 0$. Furthermore, $A(f_1,{{\bf{v}}},\xi)=1$ if and only if $\xi_1,\xi_2 < 0$ and $A(f_2,{{\bf{v}}},\xi)=1$ if and only if $\xi_2,\xi_3 < 0$.

We are only interested in the change of $i({{\bf{v}}},\xi)$ when $c_e=c_{e_2}$ is crossed, so when $\xi_2$ changes sign. Looking at the different cases of signs for $\xi_1,\xi_3$, we obtain that if $\xi_2 >0$ becomes $\xi_2 <0$, then the index decreases by $1$ if $\xi_1,\xi_3>0$, it increases by $1$ if $\xi_1,\xi_3<0$, and it does not change if $\xi_1\xi_3<0$. This means that the index decreases by $1$ if we cross $a$ (and hence $g({{\bf{v}}})$) from left to right, it increases by $1$ if we cross $-a$ (and hence $-g({{\bf{v}}})$) from left to right, and the index does not change if we cross $c_e$ at a point that does not belong to $a$ or $-a$ (and hence neither belongs to $g({{\bf{v}}})$ or $-g({{\bf{v}}})$).

To complete our proof, we have to discuss the case that ${{\bf{v}}}_3$ has coordinates $(1,0,0)$. Then, the normals of $f_1$ and $f_2$ are the previous vectors $-{\bf{n}}_2$ and $-{\bf{n}}_1$, respectively. Observing the change in orientation, we again get the same changes of the index when $c_e$ is crossed.
\end{proof}

\begin{remark}
Note that the restriction to a particular case of positions of the vertices ${{\bf{v}}},{{\bf{v}}}_1,{{\bf{v}}}_2,{{\bf{v}}}_3$ in the proof of Proposition~\ref{prop:winding} serves merely as a simplification of notation. The general idea is the same for all positions of vertices.
\end{remark}


\subsection{Proof of the main theorem}\label{sec:proof}

We are now ready to combine the previous results into a complete proof of Theorem~\ref{th:egregium}.

\begin{proof}
By definition, the algebraic area of $g({{\bf{v}}})$ equals $A(g({{\bf{v}}})):=\int_{S^2}w(g({{\bf{v}}}),\xi)d\xi=\int_{S^2}w(g({{\bf{v}}}),-\xi)d\xi$ due to antipodal symmetry. By Propositions~\ref{prop:density} and~\ref{prop:winding}, $i({{\bf{v}}},\xi)=w(g({{\bf{v}}}),\xi)+w(g({{\bf{v}}}),-\xi)$ for the dense subset of $\xi \in S^2$general for ${{\bf{v}}}$. Using $\frac{1}{2}\int_{S^2}i({{\bf{v}}},\xi)d\xi=K({{\bf{v}}})$ by Proposition~\ref{prop:integration_index}, we finally obtain \[A(g({{\bf{v}}}))=\int_{S^2}w(g({{\bf{v}}}),\xi)d\xi=\frac{1}{2}\left(\int_{S^2}w(g({{\bf{v}}}),\xi)d\xi+\int_{S^2}w(g({{\bf{v}}}),-\xi)\right)d\xi=\frac{1}{2}\int_{S^2}i({{\bf{v}}},\xi)d\xi=K({{\bf{v}}}).\qedhere\]
\end{proof}


\section{Generalization to non-embedded polyhedral surfaces and degree theory}\label{sec:degree}

In this section, we generalize Theorem~\ref{th:egregium} to polyhedral vertex stars that are not necessarily embedded. Before coming to this in Section~\ref{sec:generalization}, we introduce the normal degree of the Gauss image in Section~\ref{sec:degree_index}. The normal degree is closely related to the critical point index and the winding numbers discussed in Section~\ref{sec:winding}. We conclude in Section~\ref{sec:degree_embedded} with a brief discussion of degree theory for embedded polyhedral surfaces. The observation that the absolute value of the index is always greater or equal than the one of the degree leads us to our discussion of possible shapes of Gauss images in the forthcoming Section~\ref{sec:shape_analysis}.

As in the first remark of Section~\ref{sec:winding}, we project the polyhedral vertex star to a ball centered at the central vertex and consider spherical polygons instead. The Gauss image of a vertex star then corresponds to the polar spherical polygon.

Let $W=(w_1,w_2,\ldots,w_n)$, $w_{n+1}=w_1$, be a spherical polygon in \textit{general position}, i.e., no three vertices of $W$ lie on a common great circle. All spherical polygons we consider have edges of length less than $\pi$, so an edge between two adjacent vertices is the shortest great circle arc connecting them. In Sections~\ref{sec:degree_index} and~\ref{sec:generalization}, $W$ may intersect itself. Due to the assumption of $W$ being in general position, all intersections are transversal. Let $c(W)$ be the number of crossings of $W$ with itself.

\begin{definition}
For a spherical polygon $W=(w_1,w_2,\ldots,w_n)$ in general position, let $W'=(w'_1,w'_2,\ldots,w'_n)$ be the \textit{polar polygon} defined by the condition that \[w'_i=\frac{w_i \times w_{i+1}}{\Vert w_i \times w_{i+1} \Vert}.\]
\end{definition}

\begin{remark}
If a (not necessarily embedded) polyhedral vertex star projects to the spherical polygon $W$, then its Gauss image is identical with the polar polygon $W'$. In particular, if $\xi \in S^2$ is general for the polyhedral vertex star, then it does not lie on any great circle passing through a circular edge of $W'$. In this sense, we can speak also about \textit{generality} of $\xi \in S^2$ for the spherical polygon $W$.
\end{remark}

The definition of the critical point index $i({{\bf{v}}},\xi)$ in Section~\ref{sec:critical} did not depend on the star of ${{\bf{v}}}$ being embedded. Hence, we can define the \textit{index} $i(W,\xi)$ for a spherical polygon $W$ and a general $\xi\in S^2$ in exactly the same way.


\subsection{The normal degree of the Gauss image and its relation to winding numbers and the critical point index}\label{sec:degree_index}

\begin{definition}
Let $\xi \in S^2$ be general for the spherical polygon $W$, and let $\gamma$ be an oriented great circle arc going from $\xi$ to $-\xi$ which meets no vertex of $W'$. Then, we define $d(W',\gamma):=\sum_{i=1}^n d(w'_iw'_{i+1}, \gamma)$, where  \[d(w'_iw'_{i+1}, \gamma)=\left\{\begin{array}{cl} \frac{\langle w'_i \times w'_{i+1},\xi \rangle}{\vert \langle w'_i \times w'_{i+1}, \xi \rangle \vert} & \mbox{if }w'_iw'_{i+1} \textnormal{ intersects } \gamma,\\ 0 & \textnormal{otherwise.} \end{array}\right.\]
\end{definition}

The expression above tells us whether the edge of the polar polygon $W'$ passes $\gamma$ from right to left ($d(w'_iw'_{i+1}, \gamma)=1$) or from left to right ($d(w'_iw'_{i+1}, \gamma)=1$) or not at all ($d(w'_iw'_{i+1}, \gamma)=0$).

\begin{lemma}\label{lem:degree_independence}
For a given spherical polygon $W$, $d(W',\gamma)$ does only depend on $\xi \in S^2$ and not on the particular choice of $\gamma$.
\end{lemma}
\begin{proof}
Any admissible $\gamma$ can be seen as an (oriented) meridian of longitude, considering $\xi$ as a pole. When we move $\gamma$ around by continuously increasing the longitude, we encounter vertices $w'_i$ in two possible ways: Either $w'_{i-1}$ and $w'_{i+1}$ are on the same side of $\gamma$ or not. In the first case, we add (or subtract) $d(w'_{i-1}w'_{i}, \gamma)+d(w'_iw'_{i+1}, \gamma)$ to (or from) $d(W',\gamma)$, but $d(w'_{i-1}w'_{i}, \gamma)=-d(w'_iw'_{i+1}, \gamma)$. In the second case, $d(w'_{i-1}w'_{i}, \gamma)=d(w'_iw'_{i+1}, \gamma)$ and $d(W',\gamma)$ changes by $d(w'_{i-1}w'_{i}, \gamma)-d(w'_iw'_{i+1}, \gamma)$. In either case, $d(W',\gamma)$ does not change.
\end{proof}

\begin{definition}
For $\xi \in S^2$ general for $W$, we define the \textit{normal degree} $d(W,\xi):=d(W',\gamma)$, where $\gamma$ is any oriented great circle arc going from $\xi$ to $-\xi$.
\end{definition}

\begin{remark}
Clearly, $d(W,-\xi)=-d(W,\xi)$.
\end{remark}

\begin{proposition}\label{prop:deg_ind_cross}
Let $W=W(0)$ be a spherical polygon in general position and $W(t)$ a continuous deformation such that at no point $t$ of time, two vertices of $W(t)$ are the same or antipodal, and such that $W(t)$ is in general position for almost all $t$. Let $\xi \in S^2$ be general for $W$ and $W(t)$ for almost all $t$. Then, \[d(W(t),\xi)+i(W(t),\xi)+c(W(t)) \equiv d(W,\xi)+i(W,\xi)+c(W) \mod 2.\]
\end{proposition}
\begin{proof}
We first investigate these situations where also the polar polygon $W'(t)$ changes continuously. A discontinuity can only appear when two adjacent vertices of the polar become antipodal, which only happens when two adjacent edges of $W(t)$ overlap. So let us first consider these changes where no two edges of $W(t)$ overlap. Then, the only possibility to generate or to eliminate crossings is to pass a vertex $w_{i+1}(t)$ through an edge different from $w_{i-1}(t)w_{i}(t)$ and $w_{i+2}(t)w_{i+3}(t)$. It follows that if the number of crossings changes, then either none or both edges incident to $w_{i+1}(t)$ will intersect the edge through which $w_{i+1}(t)$ passes. So the number of crossings does not change modulo two.

The normal degree $d(W(t),\xi)$ cannot change unless $W'(t)$ passes through $\pm \xi$, and an edge $w'_i(t)w'_{i+1}(t)$ contains $\pm \xi$ if and only if $\langle w_{i+1}(t), \xi \rangle =0$, i.e., if $w_{i+1}(t)$ lies on the great circle with pole $\xi$. Let us recall Proposition~\ref{prop:index_equality} describing $i(W(t),\xi)$ as one minus half of the number of triangles $(0,w_i(t),w_{i+1}(t))$ such that $0$ is middle with respect to the height function $\langle \xi, \cdot \rangle$. Then, we see that also the critical point index does not change unless a vertex passes through the great circle with pole $\xi$.

So let us consider the changes in $d(W(t),\xi)$ and $i(W(t),\xi)$ as a vertex $w_{i+1}(t)$ passes through that great circle. If both $w_{i}(t)$ and $w_{i+2}(t)$ lie in the same hemisphere with pole $\xi$ or $-\xi$, then we degenerate or eliminate two triangles for which $0$ is middle, meaning that $i(W(t),\xi)$ changes by $\pm 1$. The normal degree $d(W(t),\xi)$ also changes by $1$ since $w'_i(t)w'_{i+1}(t)$ is moved across $\pm \xi$. If $w_{i}(t)$ and $w_{i+2}(t)$ lie in different hemispheres, then neither the index nor the degree change. Thus, $d(W(t),\xi)+i(W(t),\xi)+c(W(t)) \mod 2$ does not change when $W'(t)$ changes continuously.

\begin{figure}[!ht]
	\centerline{
		\begin{overpic}[height=0.3\textwidth]{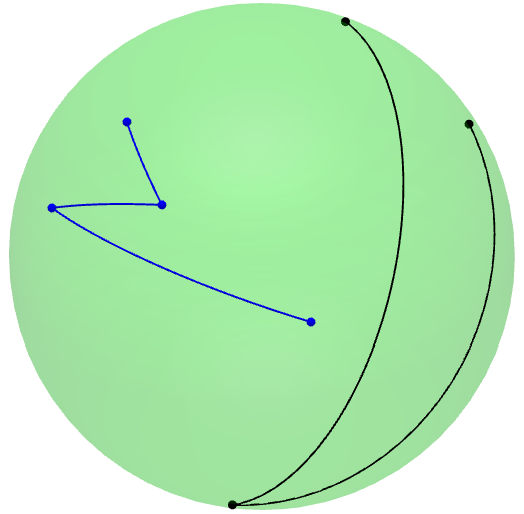}
		  \cput(98,72){{$w'_{i-1}$}}
			\cput(40,0){{$w'_{i}$}}
			\cput(58,93){{$w'_{i+1}$}}
			\color{blue}
			\cput(26,77){{$w_{i-1}$}}
			\cput(35,60){{$w_{i}$}}
			\cput(5,54){{$w_{i+1}$}}
			\cput(59,33){{$w_{i+2}$}}
		\end{overpic}
		\relax\\
		\begin{overpic}[height=0.3\textwidth]{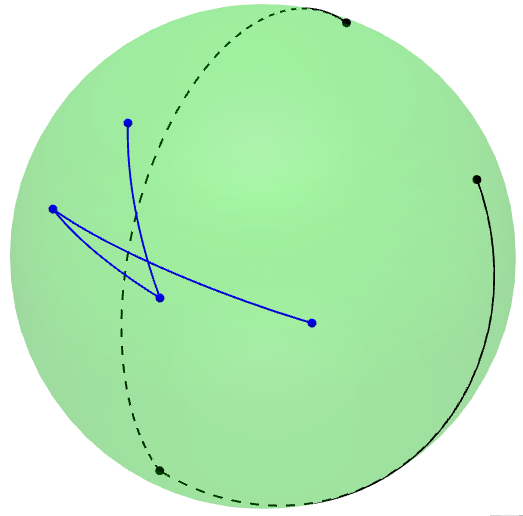}
		  \cput(98,66){{$w'_{i-1}$}}
			\cput(27,6){{$w'_{i}$}}
			\cput(58,93){{$w'_{i+1}$}}
			\color{blue}
			\cput(26,77){{$w_{i-1}$}}
			\cput(31,35){{$w_{i}$}}
			\cput(5,54){{$w_{i+1}$}}
			\cput(59,33){{$w_{i+2}$}}
		\end{overpic}
	}
	\caption{Folding of the edge $w_iw_{i+1}$ to one of its adjacent edges and the corresponding change in $W'$}\label{fig:folding}
\end{figure}

We are left with the situation that an edge is folded onto one of its adjacent edges, see Figure~\ref{fig:folding}. Then, the corresponding polar vectors will be antipodal and the degree $d(W(t),\xi)$ changes by $\pm 1$, as does the number of crossings. The index remains invariant under such a move.
\end{proof}

In the setup of embedded spherical polygons $W$ that we discussed in Section~\ref{sec:Gauss}, the number of crossings was always zero. If we think about the polyhedral vertex star that corresponds to a convex corner as in Figure~\ref{fig:convex}, then $i(W,\xi)=1=d(W,\xi)$ if $\xi$ is contained in the region where the winding number equals one. It hence follows from Proposition~\ref{prop:deg_ind_cross} that $i(W,\xi) \equiv d(W,\xi) \mod 2$ for all $W$ without self-intersections.

In the proof of Proposition~\ref{prop:deg_ind_cross} we have seen that the polar polygon $W'$ does not deform continuously under a continuous deformation of $W$ if and only if an edge is folded onto another, so if and only if $W$ fails to be an immersion of a circle into the sphere at some point of time. We are now interested in deformations that are both continuous for $W$ and $W'$, and we hence consider regular homotopies of $W$.

Let $x_0:=(m,\vec{t})$ be a point of the unit tangent bundle $T$ of a Riemannian manifold $M$. Let us consider the set of regular homotopy classes of regular closed curves on $M$ that start and end at $m$ and whose tangent vector is $\vec{t}$. By a theorem of Smale \cite{S58}, this set is in one-to-one-correspondence with $\pi_1(T,x_0)$. This theorem easily generalizes to locally regular curves, in particular polygons. In the case of the homogeneous sphere, we may skip the base point $x_0$. A point in the unit tangent bundle of $S^2$ is composed of a vector $\xi \in S^2$ and a vector $\nu \in S^2$ orthogonal to $\xi$. There is then a unique third vector $\mu \in S^2$ such that $(\xi,\nu,\mu)$ determines an element of $SO(3)$. This indicates why the unit tangent bundle of $S^2$ is diffeomorphic to $SO(3)$, which in turn is diffeomorphic to $\mathds{RP}^3$. The fundamental group of the latter is known to be $\mathds{Z}_2$. Observing that the number of crossings of $W$ modulo two is invariant under regular homotopies, we end up with the following corollary of Proposition~\ref{prop:deg_ind_cross}:

\begin{corollary}\label{cor:classes}
Under regular homotopies, there are exactly two classes of spherical polygons $W$ and they are classified by their number of crossings modulo $2$. In each class, $d(W,\xi)+i(W,\xi) \equiv c(W) \mod 2$.
\end{corollary}
\begin{remark}
Note that the regular homotopies we consider include the insertion of additional vertices on edges and the removal of vertices that lie on the great circle arc connecting its two neighbors. In the polar polygon, this corresponds to split a vertex into two or to merge two vertices into one.
\end{remark}

Using the continuity of the polar polygon $W'$ under regular homotopies, we are now able to update the winding numbers of Section~\ref{sec:winding}:

\begin{definition}
Let $W$ be an oriented spherical polygon in general position. The \textit{winding number} $w(W',\xi)$ of the polar polygon $W'$ around a point $\xi \in S^2$ that does not lie on $W'$ is the value of the unique function satisfying the following properties:
\begin{enumerate}
\item In any connected component in the complement of $W'$, $w(W',\cdot)$ is constant.
\item If $\xi$ crosses $W'$ from left to right, the winding number decreases by 1.
\item Under regular homotopies of $W$, the winding number $w(W',\xi)$ does not change as long $\xi$ never lies on $W'$.
\item If $W$ is a convex spherical polygon, $w(W',\xi)=1$ if $\xi$ is contained inside the region that $W'$ encloses (of area less than $2\pi$) as in Figure~\ref{fig:convex}. If $W$ is a locally convex spherical polygon of turning number 2, $w(W',\xi)=2$ if $\xi$ is contained inside the region that $W'$ encloses twice in counterclockwise direction, see Figure~\ref{fig:turning}.
\end{enumerate}
\end{definition}

\begin{figure}[!ht]
	\centerline{
		\begin{overpic}[height=0.3\textwidth]{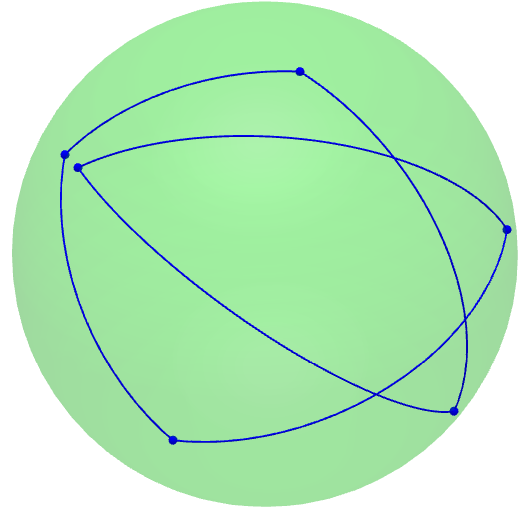}
		  \color{blue}
		  \cput(30,7){{$w_1$}}
			\cput(96,55){{$w_2$}}
			\cput(19,62){{$w_3$}}
			\cput(2,67){{$w_4$}}
			\cput(86,13){{$w_5$}}
			\cput(55,86){{$w_6$}}
		\end{overpic}
		\relax\\
		\begin{overpic}[height=0.3\textwidth]{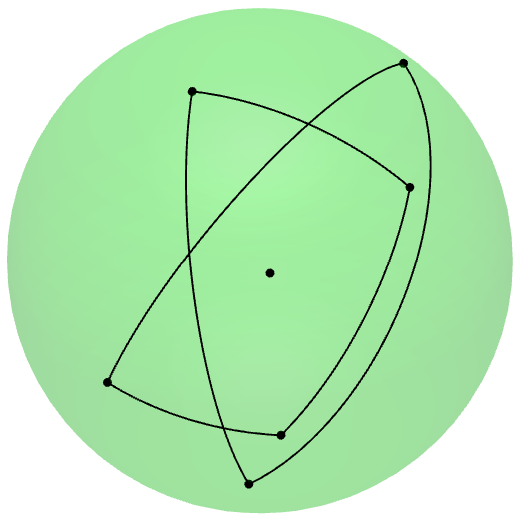}
		  \cput(31,79){{$w'_1$}}
			\cput(43,2){{$w'_2$}}
			\cput(75,91){{$w'_3$}}
			\cput(16,25){{$w'_4$}}
			\cput(50,19){{$w'_5$}}
			\cput(73,60){{$w'_6$}}
			\cput(49,48){{$\xi$}}
		\end{overpic}
	}
	\caption{$w(W',\xi)=2$ if $W'$ winds twice around $\xi$}\label{fig:turning}
\end{figure}

\begin{remark}
Note that $c(W)=3 \equiv 1 \mod 2$ in Figure~\ref{fig:turning}. So by Corollary~\ref{cor:classes}, we defined the winding number for both regular homotopy classes of spherical polygons.
\end{remark}

We are now ready to generalize Proposition~\ref{prop:winding} comparing the index and the winding numbers to general polygons, and to formulate a corresponding statement for the normal degree. This is again the key ingredient of the proof of the (generalization of) Theorem~\ref{th:egregium}.
\begin{proposition}\label{prop:degree_winding}
Let $W$ be a spherical polygon and let $\xi\in S^2$ be general for $W$. Furthermore, let $c:=0$ if the number of self-intersections of $W$ is even, and $:c:=1$ if $c(W)$ is odd. Then, \begin{align}c+i(W,\xi)&=w(W',\xi)+w(W',-\xi), \label{eq:winding2}\\ d(W,\xi)&=w(W',\xi)-w(W',-\xi). \label{eq:degree}\end{align}
\end{proposition}
\begin{proof}
In exactly the same way as in the proof of Proposition~\ref{prop:winding} it follows that both sides of Equation~(\ref{eq:winding2}) agree up to a constant for each regular homotopy class. Thus, it suffices to check that equation for just one particular case for each class.

The case $c=0$ was already considered in the proof of Proposition~\ref{prop:winding}, so we only have to check the case $c=1$ that is depicted in Figure~\ref{fig:turning}. There, $w(W',\xi)=2$ and it is easy to see that $w(W',-\xi)=0$. On the other hand, $\langle \xi, w_i \rangle >0$ for all $i$, such that $i(W,\xi)=1$. In particular, $1+i(W,\xi)=w(W',\xi)+w(W',-\xi)$ as claimed.

To check Equation~(\ref{eq:degree}) we can argue in a similar way. First, we have to show that both $d(W,\xi)$ and $w(W',\xi)-w(W',-\xi)$ behave the same when $\xi$ is moved. It follows from the definition of the degree that it changes only when $\xi$ or $-\xi$ passes across $W'$. If $\xi$ passes $W'$ from left to right, we can choose a great circle arc from $\xi$ to $-\xi$ that is crossed by $W'$ from right to left before the movement of $\xi$ and that is not intersected afterward. Hence, $d(W,\xi)$ decreases by $1$. By a similar argument, it decreases by $1$ if $-\xi$ crosses $W'$ from right to left. Therefore, both sides of Equation~(\ref{eq:degree}) behave the same under changing $\xi$.

In the case $c=0$, we have again a look at Figure~\ref{fig:convex}. Where $w(W',\xi)=1$, we have $w(W',-\xi)=0$ and $d(W,\xi)=1=w(W',\xi)-w(W',-\xi)$. If $c=1$, then we consider Figure~\ref{fig:turning}. There, $w(W',\xi)=2$, $w(W',-\xi)=0$, and $d(W,\xi)=2=w(W',\xi)-w(W',-\xi)$.
\end{proof}

\begin{corollary}\label{cor:deg_index}
Let $W$ be a simple spherical polygon and let $\xi$ be general. Then, $i(W,\xi) \equiv d(W,\xi) \mod 2$. If there exists a great circle such that $W$ projects along perpendicular great circle arcs onto it in a one-to-one way, then $i(W,\xi) =\pm d(W,\xi)$.
\end{corollary}
\begin{proof}
The first statement follows directly from Corollary~\ref{cor:classes} using $c(W)=0$. For the more specific claim $i(W,\xi) =\pm d(W,\xi)$, we observe that existence of a great circle onto which $W$ projects along perpendicular great circle arcs in a one-to-one way corresponds to the existence of a transverse plane for the corresponding vertex star centered at the origin. So by Proposition~\ref{prop:transverse}, $W'$ is contained in an open hemisphere. If we can show that $w(W',\xi)=0$ if $\xi$ is not contained in that hemisphere, then it follows that regardless of the position of $\xi$, at least one of $w(W',\xi)$ and $w(W',-\xi)$ is zero. Proposition~\ref{prop:degree_winding} then implies $i(W,\xi) =\pm d(W,\xi)$.

To prove that $w(W',\xi)=0$ if $\xi$ is not contained in that hemisphere, we consider a certain regular homotopy to a convex spherical polygon $W(1)$. Let $\nu$ be a pole of the great circle onto which $W$ projects bijectively. Our assumptions assure that any great circle arc from $\nu$ to $-\nu$ intersects $W$ exactly once. We can now move the points of $W$ along the corresponding great circle arcs from $\nu$ to $-\nu$ to finally obtain a convex spherical polygon. Throughout the whole process, the polar polygon remains in an open hemisphere with pole $\pm \nu$. For the convex case, we know by definition that $w(W'(1),\xi)=0$ if $\xi$ is not contained in the region $W'(1)$ encloses.
\end{proof}


\subsection{Generalization of the main theorem to non-embedded polyhedral surfaces}\label{sec:generalization}

As a conclusion of our discussion in Section~\ref{sec:degree_index}, we now prove the following generalization of Theorem~\ref{th:egregium}. Note that the discrete Gaussian curvature of the star around a vertex ${{\bf{v}}}$ agrees with $2\pi$ minus the circumference of the spherical polygon defined by the intersection of the face cones of the vertex star with the unit sphere.

\begin{theorem}\label{th:egregium2}
Let $W$ be a spherical polygon in general position and let $c:=0$ if the number of self-intersections of $W$ is even, and $c:=1$ if it is odd. Then: \[2\pi(1+c)-\textnormal{length}(W)=\int\limits_{S^2}w(W',\xi)d\xi.\]
\end{theorem}
\begin{proof}
By definition, the algebraic area of $W'$ equals $A(W'):=\int_{S^2}w(W',\xi)d\xi=\int_{S^2}w(W',-\xi)d\xi$ due to antipodal symmetry. Proposition~\ref{prop:density} easily generalizes to the current situation and asserts that almost all $\xi$ are general for $W$. It then follows from Proposition~\ref{prop:degree_winding} that \begin{align*}A(W')=\int_{S^2}w(W',\xi)d\xi=\frac{1}{2}\left(\int_{S^2}w(W',\xi)d\xi+\int_{S^2}w(W',-\xi)\right)d\xi&=\frac{1}{2}\int_{S^2}\left(i(W,\xi)+c\right)d\xi\\&=2\pi-\textnormal{length}(W)+2\pi c.\end{align*}
For the last step, we remark that the proof of Proposition~\ref{prop:integration_index} stating that the integrated index gives twice the discrete Gaussian curvature of the vertex star did not make use of the vertex star being embedded. Hence, the proof of this proposition literally translates to the situation of a spherical polygon with self-intersections.
\end{proof}

If we forget about the fact that there are only two regular homotopy classes of spherical polygons, then it seems to be surprising that we distinguish only two cases in Theorem~\ref{th:egregium2}, namely whether the number of self-intersections of the spherical polygon is even or odd. In Figure~\ref{fig:turning}, we considered a polygon winding twice around itself as its polar polygon does. To obtain the algebraic area of the polar polygon that agrees with our intuition, we had to add one to the index. If we considered a polygon winding three times around itself as in Figure~\ref{fig:counterintuitive1}, we would expect that we have to add two to the index to get the correct algebraic area. But instead of $w(W',\xi)=3$ and $w(W',-\xi)=0$, we have $w(W',\xi)=2$ and $w(W',-\xi)=-1$, so $i(W,\xi)=1=2-1=w(W',\xi)+w(W',-\xi)$. The reason for this counterintuitive behavior is our requirement that the winding number should not change under regular homotopies. By Corollary~\ref{cor:classes}, the polygon $W$ depicted in Figure~\ref{fig:counterintuitive1} is regular homotopic to a convex spherical polygon.

\begin{figure}[!ht]
	\centerline{
		\begin{overpic}[height=0.3\textwidth]{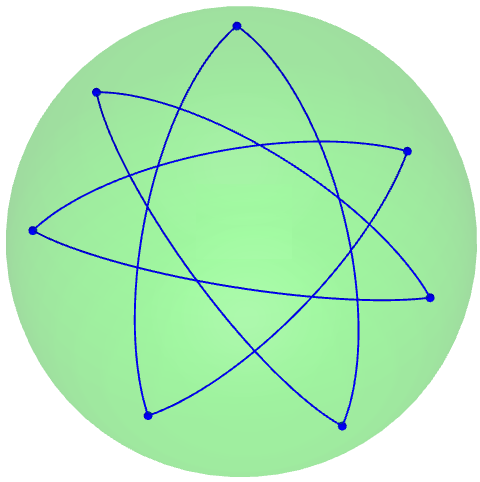}
		  \color{blue}
			\cput(82,70){{$w_1$}}
			\cput(0,48){{$w_2$}}
			\cput(86,32){{$w_3$}}
			\cput(9,80){{$w_4$}}
			\cput(70,7){{$w_5$}}
			\cput(47,97){{$w_6$}}
			\cput(28,8){{$w_7$}}
		\end{overpic}
		\relax\\
		\begin{overpic}[height=0.3\textwidth]{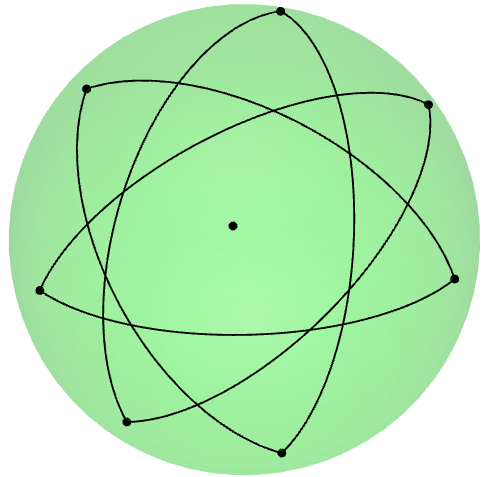}
		  \cput(59,0){{$w'_1$}}
			\cput(60,97){{$w'_2$}}
			\cput(22,6){{$w'_3$}}
			\cput(93,74){{$w'_4$}}
			\cput(7,32){{$w'_5$}}
			\cput(99,40){{$w'_6$}}
			\cput(12,79){{$w'_7$}}
			\cput(50,48){{$\xi$}}
		\end{overpic}
	}
	\caption{$w(W',\xi)=2$ and $w(W',-\xi)=-1$ even though $W'$ winds three times around $\xi$}\label{fig:counterintuitive1}
\end{figure}

In the following sequence of pictures, we indicate the regular homotopy that deforms $W$ into a convex spherical polygon and also stress the point when our intuitive winding number makes a jump and finally agrees with the winding number we defined.

In the first step, we move $w_2$ over the back of the sphere closer to $w_1$ and $w_3$. Note that moving $w_2$ over the top of the sphere would not correspond to a regular homotopy since we would then fold the edge $w_1w_2$ over $w_1w_7$ and the edge $w_2w_3$ over $w_3w_4$. Afterwards, we move $w_4$ closer to $w_5$. 

\begin{figure}[!ht]
	\centerline{
		\begin{overpic}[height=0.3\textwidth]{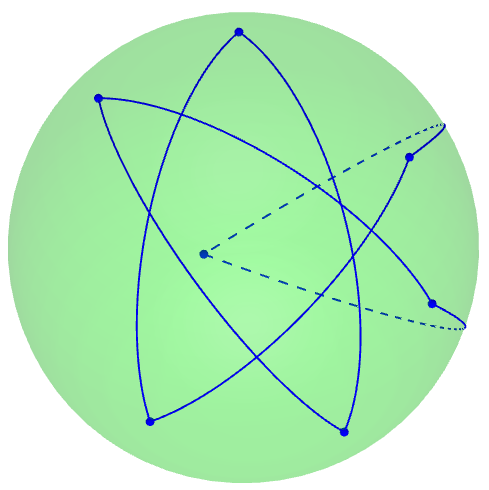}
		  \color{blue}
			\cput(83,65){{$w_1$}}
			\cput(45,47){{$w_2$}}
			\cput(87,39){{$w_3$}}
			\cput(9,80){{$w_4$}}
			\cput(70,7){{$w_5$}}
			\cput(47,96){{$w_6$}}
			\cput(28,8){{$w_7$}}
		\end{overpic}
		\relax\\
		\begin{overpic}[height=0.3\textwidth]{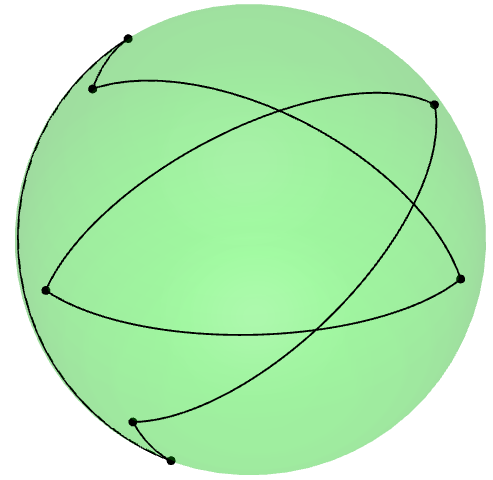}
		  \cput(26,91){{$w'_1$}}
			\cput(39,3){{$w'_2$}}
			\cput(22,14){{$w'_3$}}
			\cput(93,74){{$w'_4$}}
			\cput(10,31){{$w'_5$}}
			\cput(99,40){{$w'_6$}}
			\cput(18,73){{$w'_7$}}
		\end{overpic}
	}
	\centerline{
		\begin{overpic}[height=0.3\textwidth]{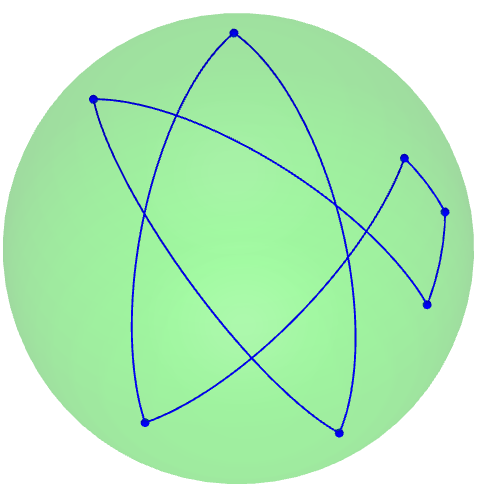}
		  \color{blue}
			\cput(81,70){{$w_1$}}
			\cput(90,55){{$w_2$}}
			\cput(86,32){{$w_3$}}
			\cput(9,80){{$w_4$}}
			\cput(65,7){{$w_5$}}
			\cput(46,95){{$w_6$}}
			\cput(28,8){{$w_7$}}
		\end{overpic}
		\relax\\
		\begin{overpic}[height=0.3\textwidth]{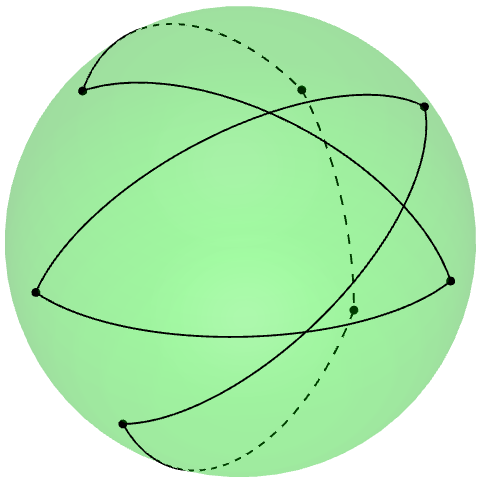}
		  \cput(64,83){{$w'_1$}}
			\cput(78,38){{$w'_2$}}
			\cput(22,6){{$w'_3$}}
			\cput(93,77){{$w'_4$}}
			\cput(7,32){{$w'_5$}}
			\cput(99,40){{$w'_6$}}
			\cput(12,79){{$w'_7$}}
		\end{overpic}
	}
	\centerline{
		\begin{overpic}[height=0.3\textwidth]{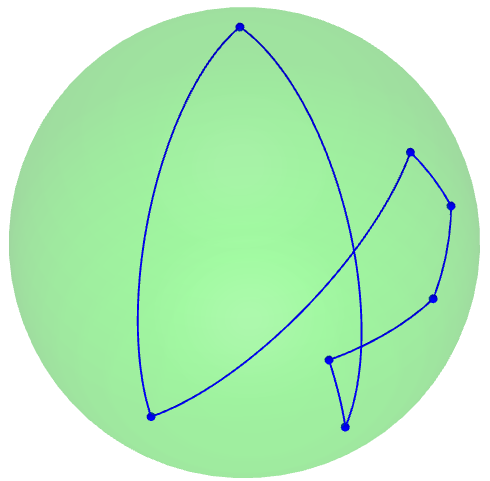}
		  \color{blue}
			\cput(80,69){{$w_1$}}
			\cput(93,54){{$w_2$}}
			\cput(86,32){{$w_3$}}
			\cput(57,24){{$w_4$}}
			\cput(68,6){{$w_5$}}
			\cput(45,95){{$w_6$}}
			\cput(28,8){{$w_7$}}
		\end{overpic}
		\relax\\
		\begin{overpic}[height=0.3\textwidth]{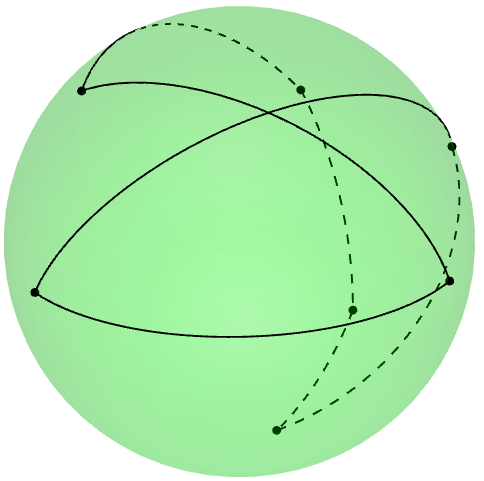}
		  \cput(64,83){{$w'_1$}}
			\cput(66,34){{$w'_2$}}
			\cput(52,9){{$w'_3$}}
			\cput(98,68){{$w'_4$}}
			\cput(7,32){{$w'_5$}}
			\cput(99,40){{$w'_6$}}
			\cput(12,79){{$w'_7$}}
		\end{overpic}
	}
	\caption{Regular homotopy to a convex spherical polygon (part one)}\label{fig:counterintuitive_i}
\end{figure}

In the polar polygon in the upper picture of Figure~\ref{fig:counterintuitive_i}, our intuition would assign winding numbers two, one, and zero to the different components of the sphere, thinking about a star polygon turning twice. The polar polygon in the lower picture resembles a figure 8, so our intuition would assign winding numbers one, zero, and minus one to the complementary regions. This now coincides with our definition of the winding numbe. The turning point is depicted in the middle picture, where we move the polyarc $w'_7w'_1w'_2w'_3$ to the back of the sphere. In this picture, it is hard to assign winding numbers intuitively and both possibilities seem to be appropriate.

We continue our deformation from $W$ to a convex spherical polygon by moving $w_6$ over the back of the sphere closer to $w_5$ and $w_7$. Finally, we move $w_4$ (and slightly $w_3$) to make the spherical polygon convex, see Figure~\ref{fig:counterintuitive_ii}.

\begin{figure}[!ht]
	\centerline{
		\begin{overpic}[height=0.3\textwidth]{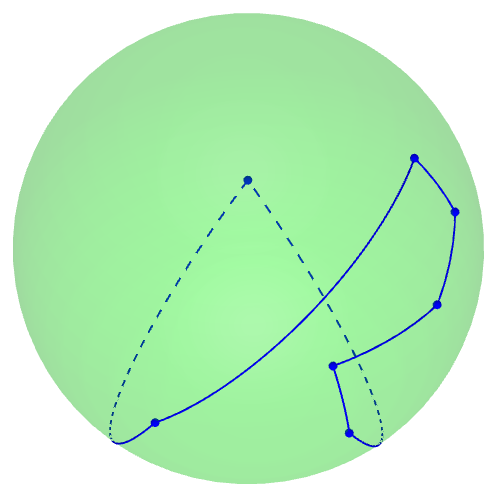}
		  \color{blue}
			\cput(82,69){{$w_1$}}
			\cput(93,55){{$w_2$}}
			\cput(87,32){{$w_3$}}
			\cput(57,24){{$w_4$}}
			\cput(61,10){{$w_5$}}
			\cput(47,65){{$w_6$}}
			\cput(29,9){{$w_7$}}
		\end{overpic}
		\relax\\
		\begin{overpic}[height=0.3\textwidth]{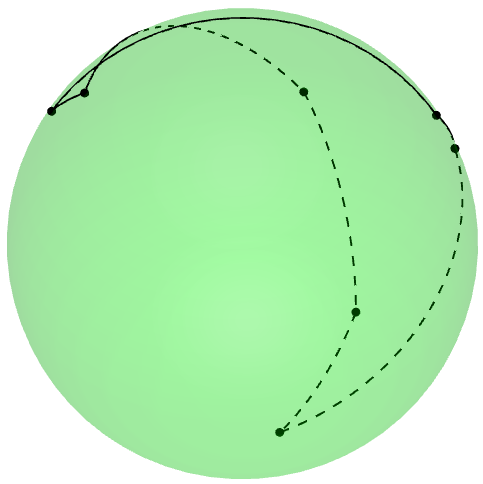}
		  \cput(64,83){{$w'_1$}}
			\cput(68,34){{$w'_2$}}
			\cput(53,9){{$w'_3$}}
			\cput(99,68){{$w'_4$}}
			\cput(92,78){{$w'_5$}}
			\cput(6,75){{$w'_6$}}
			\cput(23,79){{$w'_7$}}
		\end{overpic}
	}
	\centerline{
		\begin{overpic}[height=0.3\textwidth]{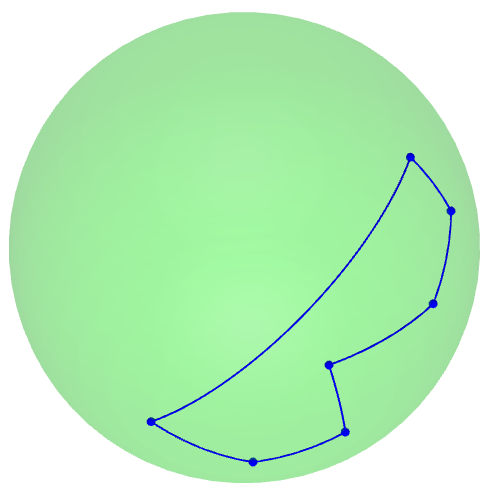}
		  \color{blue}
			\cput(80,70){{$w_1$}}
			\cput(91,55){{$w_2$}}
			\cput(88,36){{$w_3$}}
			\cput(57,24){{$w_4$}}
			\cput(70,11){{$w_5$}}
			\cput(47,1){{$w_6$}}
			\cput(24,10){{$w_7$}}
		\end{overpic}
		\relax\\
		\begin{overpic}[height=0.3\textwidth]{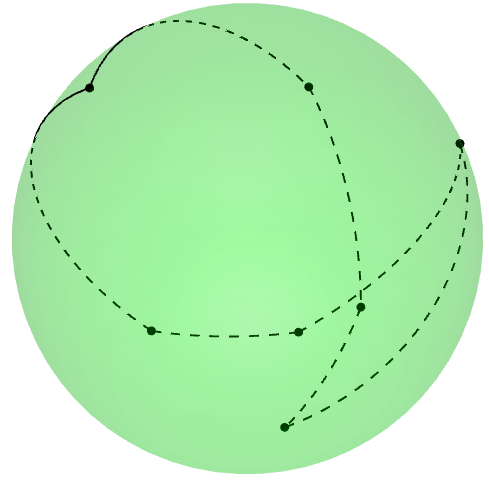}
		  \cput(65,82){{$w'_1$}}
			\cput(79,34){{$w'_2$}}
			\cput(53,9){{$w'_3$}}
			\cput(99,68){{$w'_4$}}
			\cput(58,33){{$w'_5$}}
			\cput(32,33){{$w'_6$}}
			\cput(23,79){{$w'_7$}}
		\end{overpic}
	}
	\centerline{
		\begin{overpic}[height=0.3\textwidth]{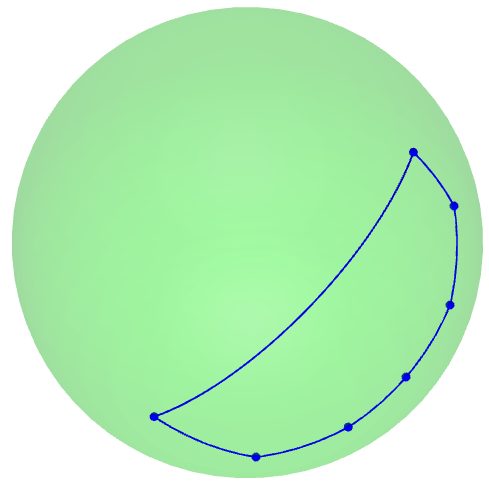}
		  \color{blue}
			\cput(82,70){{$w_1$}}
			\cput(93,54){{$w_2$}}
			\cput(92,35){{$w_3$}}
			\cput(84,20){{$w_4$}}
			\cput(70,7){{$w_5$}}
			\cput(47,1){{$w_6$}}
			\cput(27,8){{$w_7$}}
		\end{overpic}
		\relax\\
		\begin{overpic}[height=0.3\textwidth]{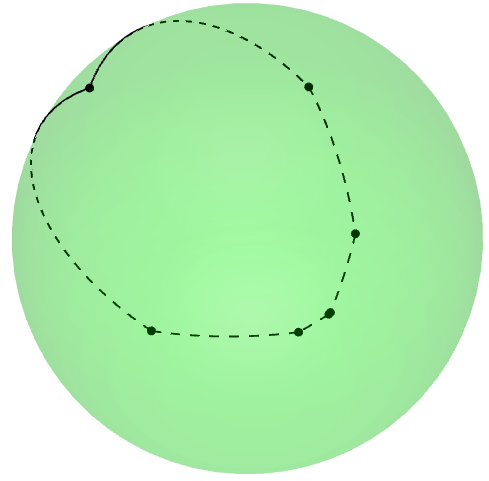}
		  \cput(65,82){{$w'_1$}}
			\cput(78,49){{$w'_2$}}
			\cput(73,35){{$w'_3$}}
			\cput(68,29){{$w'_4$}}
			\cput(58,33){{$w'_5$}}
			\cput(32,33){{$w'_6$}}
			\cput(23,79){{$w'_7$}}
		\end{overpic}
	}
	\caption{Regular homotopy to a convex spherical polygon (part two)}\label{fig:counterintuitive_ii}
\end{figure}

If we want to assign winding numbers to the components complementary to $W'$ in Figure~\ref{fig:counterintuitive1} that agree with our intuition, we can do the following. For each positive integer $k$, we pick a locally convex spherical polygon $W_k$ that turns in total $k$ times around itself, and fix points ${\bf{n}},\xi \in S^2$ around which $W_k$ and $W'_k$ wind $k$ times, respectively. We now fix $w(W'_k,\xi)=k$ and consider only regular homotopies on $S^2 \backslash \{{\bf{n}}\}$, meaning that the normal ${\bf{n}}$ should always indicate the outside of the corresponding vertex star. By stereographic projection, $S^2 \backslash \{{\bf{n}}\}$ is homeomorphic to the plane and we can apply the Whitney-Graustein theorem classifying regular homotopy classes of immersions of a circle into the plane \cite{W37}. Two polygons are regular homotopic to each other if and only if they have the same total turning number. Thus, we can define different winding numbers for all spherical polygons with total turning number $k>0$ and these winding numbers now agree with our intuition. In the deformation shown in Figures~\ref{fig:counterintuitive_i} and~\ref{fig:counterintuitive_ii}, any point of the sphere is contained at least once in a polygon of the (continuous) deformation.

Note that negative total turning numbers correspond to a change of orientation with respect to ${\bf{n}}$, which we can resolve by replacing ${\bf{n}}$ by $-{\bf{n}}$ for the polygons $W_k$. The case of total turning number $0$ corresponds to a figure 8 for which we have to fix the winding numbers separately.

If $T\geq 0$ denotes the total turning number, Theorem~\ref{th:egregium2} then changes to \[2\pi T-\textnormal{length}(W)=\int\limits_{S^2}w(W',\xi)d\xi.\]


\subsection{Degree theory for embedded polyhedral surfaces}\label{sec:degree_embedded}

In this subsection, we consider embedded polyhedral vertex star and their Gauss images. Equivalently, we discuss simple spherical polygons $W=(w_1,w_2,\ldots,w_n)$ and their polar polygons $W'$. We will use the notation $w_0:=w_n$ and $w_{n+1}:=w_1$. Furthermore, we always assume that $w_{i+1}\neq w_i$ for all $i$. Our aim is to deduce relations between the critical point index $i(W,\xi)$ and the normal degree $d(W,\xi)$. Our main method will be the concept of a $\xi$\textit{-isotopy} for which we introduce the notion of \textit{admissibility} of normal vectors.

\begin{definition}
A normal vector $\xi \in S^2$ is said to be \textit{admissible} for the spherical polygon $W$ if $\langle \xi,w_i \rangle=0$ implies $\langle \xi,w_{i-1} \rangle \langle \xi,w_{i+1} \rangle<0$ for all $i$. Let $S(\xi)$ denote the great circle with pole $\xi$.
\end{definition}

This definition guarantees that no edge of $W$ lies on $S(\xi)$, and that $W$ can never only touch this great circle. In particular, the height function $\langle \xi, \cdot \rangle$ has a unique minimum and a unique maximum on any triangle $(0,w_i,w_{i+1})$. The index $i(W,\xi)$ hence equals $1-M$, where $M$ is the number of $w_i$ such that $\langle \xi,w_i \rangle>0$ but $\langle \xi,w_{i-1} \rangle\leq 0$. So $M$ is in fact half of the number of intersections of $W$ with $S(\xi)$.

From now on, we assume that all the points $p$ of the spherical polygon $W$ with $\langle \xi,p \rangle=0$ are vertices. Given any spherical polygon, this can be achieved by subdivision. Then, $i(W,\xi)$ equals one minus half of the number of vertices $w_i$ that satisfy $\langle \xi,w_i \rangle=0$.

\begin{definition}
Let $\xi$ be admissible for the spherical polygon $W$. A $\xi$\textit{-isotopy} of $W$ is a continuous family of simple closed polygons $W(t)=(w_1(t),w_2(t),\ldots,w_n(t))$, $t \in [0,1]$, such that $W(0)=W$ and $\langle \xi,w_i(t) \rangle=0$ if and only if $\langle \xi,w_i \rangle=0$.
\end{definition}

\begin{lemma}\label{lem:unchanged}
Throughout a $\xi$-isotopy, both the index and the normal degree remain invariant: \[i(W(t),\xi)=i(W,\xi) \textnormal{ and } d(W(t),\xi)=d(W,\xi).\]
\end{lemma}
\begin{proof}
Clearly, the number of vertices on $S(\xi)$ does not change throughout the deformation. Thus, $i(W(t),\xi)=i(W,\xi)$. We observed in the proof of Proposition~\ref{prop:deg_ind_cross} that the normal degree can only change if a vertex of $W$ passes through $S(\xi)$. Hence, $d(W(t),\xi)=d(W,\xi)$ as well.
\end{proof}

In the following, we aim at finding a suitable $\xi$-isotopy that deforms $W$ into a spherical polygon for which it is simpler to determine the normal degree.

\begin{definition}
A vertex $w_i$ of a spherical polygon $W$ is called \textit{essential}, if $w_i$ does not lie on the great circle arc connecting $w_{i-1}$ and $w_{i+1}$.
\end{definition}

\begin{lemma}\label{lem:deformation}
Let $Y=(y_0,\ldots,y_r)$ be an embedded polygon contained in a closed hemisphere, and let $R$ be the closed region bounded by $Y$ in that hemisphere. Then, there is a deformation $Y(t)=(y_0(t),\ldots,y_r(t))$, $t \in [0,1]$, such that each $Y(t)$ is an embedded polygon contained within $R$, such that $y_i(t)=y_i$ for $i=0,r$ and all times $t$, and such that $Y(1)$ has only three essential vertices.
\end{lemma}
\begin{proof}
To begin with, we may assume that all vertices $y_i$ of $Y$ are essential. To achieve that, non-essential vertices can be slightly perturbed into the region $R$. Our general procedure is finding a \textit{free} vertex $y_i$ in $Y$, $i\neq 0,r$. A free vertex $y_i$ is a vertex such that the open great circle arc $a$ from $y_{i-1}$ to $y_{i+1}$ is contained in $R$. If we have found such a free vertex $y_i$, we fix all other vertices and move $y_i$ linearly to the midpoint of $y_{i-1}y_{i+1}$. This yields a polygon with fewer essential vertices. We can now delete the vertex and replace the two incident edges by $a$. Then, we proceed by finding another free vertex in the polygon. We continue until only three essential vertices remain.

So we are done if we can show that every spherical polygon (with all vertices being essential) that has more than three vertices has at least two non-adjacent free vertices $y_i$, $y_j$, $j-i\neq \pm 1$. As required, one of them will be neither $y_0$ nor $y_r$. The proof will be by induction on the number of essential vertices.

We first show that any polygon $Y$ that has more than three (essential) vertices and that lies in a hemisphere has at least one interior diagonal. An interior diagonal is a great circle arc connecting two non-adjacent essential vertices $y_i$ and $y_k$ that lies in the region $R$ bounded by $Y$. The proof of this is a direct extension of the standard proof for simple planar polygons:

Let us rotate the equator of the hemisphere containing $Y$ around two antipodal points until it meets $Y$. Let $y_i$ be one of the vertices where this great circle meets $Y$. Then, the interior angle of $Y$ at $y_i$ is less than $\pi$.

Consider the triangle $\triangle$ formed by $y_{i-1},y_i,y_{i+1}$. If no essential vertex is contained within $\triangle$ or the great circle arc $y_{i-1}y_{i+1}$, then the latter is an interior diagonal. If no essential vertex lies within $\triangle$, but $y_k$ is an essential vertex on $y_{i-1}y_{i+1}$, then the arc $y_iy_k$ is an interior diagonal. If $\triangle$ contains essential vertices, then consider the essential vertex $y_k$ that is closest to $y_i$ (if several vertices have the minimal distance, take any of them). Clearly, the arc $y_iy_k$ is then an interior diagonal.

Secondly, we observe that the previous paragraph says in particular that a polygon in a hemisphere with exactly four essential vertices has an interior diagonal $y_{k-1}y_{k+1}$, so $y_{k-1}$ and $y_{k+1}$ are non-adjacent free vertices. If $Y$ has more than four vertices, we first perturb the vertices $y_i$, $i \neq 0,r$, into $R$ so that no three essential vertices of $Y$ lie on a great circle. By our argument in the previous paragraph, we find an interior diagonal $y_iy_j$ of $Y$ separating $Y$ into two polygons $Y_1$ and $Y_2$ with fewer essential vertices. The region $R$ splits into two regions $R_1$ and $R_2$ bounded by $Y_1$ and $Y_2$, respectively. Note that $y_i$ and $y_j$ are essential vertices of both $Y_1$ and $Y_2$. In the case that $i$ and $j$ are just one vertex apart, then one of the polygons $Y_1$ and $Y_2$ has exactly three essential vertices, and the third vertex $y_{i-1}$ or $y_{i+1}$ will be a free vertex of $Y$. If $Y_1$ or $Y_2$ has more than three essential vertices (but less essential vertices than $Y$), then, by induction, the subpolygon has two non-adjacent free vertices. One of them is neither $y_i$ nor $y_j$, and this free vertex of the subpolygon will also be a free vertex of $Y$. We conclude that $Y$ has at least two non-adjacent free vertices, which completes the proof.
\end{proof}

\begin{remark}
From the proof above it follows that all $Y(t)$ are contained within $R \backslash T(Y)$ for some triangle $T(Y)$ with base $y_0y_r$. For example, we may choose $T(Y)$ so that $T(Y)$ meets no arc $y_0y_i$ or $y_iy_r$.
\end{remark}

We are now able to construct suitable $\xi$-isotopies for spherical polygons that intersect $S(\xi)$ exactly twice, and we deduce their normal degree.

\begin{proposition}\label{prop:deformation0}
If $i(W,\xi)$=0, then there is a $\xi$-isotopy such that $W(1)$ lies in a great circle through $\xi$.
\end{proposition}

\begin{proof}
If necessary, we subdivide $W$ and rename its vertices such that the vertices $w_0$ and $w_r$ are the intersection points of $W$ with $S(\xi)$. Without loss of generality, we assume that the polyarc $w_0w_1\ldots w_r$ lies in the upper hemisphere and the polyarc $w_rw_{r+1}\ldots w_n w_0$ in the lower hemisphere. We slightly perturb the points $w_0$ and $w_r$ if necessary such that the two points are not antipodal. By Lemma~\ref{lem:deformation}, we may find a deformation of the polygon $W_1=(w_0,w_1,\ldots,w_r)$ fixing $w_0$ and $w_r$ such that $W_1(1/3)$ has exactly one other essential vertex, which we may move to $\xi$. We extend this deformation to $W$ by letting $w_j$ remain constant for $j=r,r+1,\ldots,0$. Then, we fix all the vertices in the upper hemisphere and apply Lemma~\ref{lem:deformation} to the polygon $W_2$ in the lower hemisphere to obtain a spherical triangle $W_2(2/3)=(w_r,-\xi,w_0)$.

Now, keeping $\xi$ and $-\xi$ fixed, we move $w_0$ and $w_r$ until they are antipodal. We obtain a $\xi$-isotopy to a polygon $W(1)$ that lies in the great circle through $w_0$, $w_r$, and $\pm \xi$.
\end{proof}

\begin{corollary}\label{cor:index0}
If $i(W,\xi)=0$, then $d(W,\xi)=0$.
\end{corollary}

\begin{proof}
At the end of the $\xi$-isotopy of Proposition~\ref{prop:deformation0}, the polar polygon of $W(1)$ degenerates to a point on $S(\xi)$. So we have $d(W,\xi)=d(W(1),\xi)=0$ by Lemma~\ref{lem:unchanged}.
\end{proof}

\begin{proposition}\label{prop:index1}
If $i(W,\xi)=1$, then $d(W,\xi)=\pm 1$.
\end{proposition}

\begin{proof}
Since $i(W,\xi)=1$, the intersection of $W$ with $S(\xi)$ is empty. Hence, $W$ lies in a closed hemisphere and we can apply Lemma~\ref{lem:deformation} to find a $\xi$-isotopy $W(t)$ such that $W(1)$ has precisely three essential vertices. We perturb them in such a way that $W(1)$ is a equilateral spherical triangle whose vertices have the same distance to $\xi$. For such a triangle, the polar triangle is of the same form, so $d(W,\xi)=d(W(1),\xi)=\pm 1$.
\end{proof}

\begin{corollary}\label{cor:maxmindegree}
If the height function $\langle \xi, \cdot \rangle$ has an absolute maximum (or minimum) at the central vertex of a polyhedral disc that projects to the spherical polygon $W$, then $d(W,\xi)=1$ (or $d(W,\xi)=-1$).
\end{corollary}

\begin{proof}
If $\langle \xi, \cdot \rangle$ has a maximum at the central vertex, $W$ lies in the hemisphere with pole $-\xi$. In particular, $i(W,\xi)=1$ and $d(W,\xi)=\pm 1$ by Proposition~\ref{prop:index1}. The polar polygon then lies in the hemisphere with pole $\xi$ and its orientation around $\xi$ is the same as the one of $W$. Thus, $d(W,\xi)=1$. If instead $\langle \xi, \cdot \rangle$ has a minimum at the central vertex, $\langle -\xi, \cdot \rangle$ has a maximum and we deduce the statement from $d(W,\xi)=-d(W,-\xi)=-1$.
\end{proof}

Before we will our discussion of $\xi$-isotopies, we shortly comment how the degree theory discussed above relates to height functions on surfaces with three critical points as discussed by the first author and Takens in \cite{BF75}. As in Section~\ref{sec:curvature}, let $P$ be a closed simplicial surface embedded into three-dimensional Euclidean space. For a vertex ${\bf{v}}$ of $P$ and a normal vector $\xi$ that is general for ${\bf{v}}$, let $d({\bf{v}},\xi)$ denote the normal degree of the spherical polygon obtained by projecting the polyhedral vertex star to a ball centered at ${\bf{v}}$.

\begin{proposition}\label{prop:totaldegree}
Let $\xi \in S^2$ be general for $P$. Then: \[\sum\limits_{{\bf{v}}\in V(P)} d({\bf{v}},\xi)=0.\]
\end{proposition}

\begin{proof}
Recall the definition of the normal degree given in Section~\ref{sec:degree_index}. We choose an oriented great circle arc $\gamma$ from $\xi$ to $-\xi$. An edge $\epsilon$ of the Gauss image of the star of a vertex ${\bf{v}}$ (which equals the corresponding polar polygon) contributes to the degree with $1$ or $-1$ depending on whether it intersects $\gamma$ from right to left or from left to right. The orientation of $\epsilon$ is inherited by the orientation of the vertex star. Now, an edge in the Gauss image corresponds to an edge of the polyhedral surface. If $\epsilon$ corresponds to the edge $e \in E(P)$ and $e={\bf{v}}{\bf{w}}$, then $\epsilon$ also appears in the Gauss image of the star of ${\bf{w}}$, but in reverse orientation. Since $P$ is a closed surface, all the contributions of edges to the sum of normal degrees sum up to zero.
\end{proof}

\begin{theorem}\label{th:3cp}
Let the height function $\langle \xi, \cdot \rangle$ have exactly three critical points on the polyhedral surface $P$. Let ${\bf{v}}_+$ be the vertex where the maximum is attained, ${\bf{v}}_-$ the point where $\langle \xi, {\bf{v}}_- \rangle$ is minimal, and let ${\bf{v}}_0$ denote the other critical point. Then, \[d({\bf{v}}_0,\xi)=0 \quad \textnormal{ and } \quad i({\bf{v}}_0,\xi)=\chi(P)-2.\]

In particular, there is no polyhedral embedding of a sphere such that a height function with exactly three critical points exists. Furthermore, there is no transversally embedded closed polyhedral surface with a height function with exactly three critical points. Here, an embedding is called transversal if all vertices ${\bf{v}}\in V(P)$ possess a transverse plane.
\end{theorem}

\begin{proof}
For the maximum vertex ${\bf{v}}_+$ and the minimum vertex ${\bf{v}}_-$ we have $i({\bf{v}}_+,\xi)=0=i({\bf{v}}_-,\xi)$. In addition, $d({\bf{v}}_+,\xi)=1=-d({\bf{v}}_-,\xi)$ by Corollary~\ref{cor:maxmindegree}. Since the height function has no further critical points, the index of any other vertex of $P$ is zero, and so is its normal degree according to Corollary~\ref{cor:index0}.

Since the critical point indices sum up to $\chi(P)$ by Proposition~\ref{prop:critical}, $i({\bf{v}}_0,\xi)=\chi(P)-2$. By Proposition~\ref{prop:totaldegree}, the normal degrees sum up to zero, so $d({\bf{v}}_0,\xi)=0$.

If $P$ is a polyhedral sphere, then $\chi(P)=2$ and $i({\bf{v}}_0,\xi)=0$ in contradiction to the assumption that ${\bf{v}}_0$ is a critical point. If the vertex ${\bf{v}}_0$ has a transverse plane, then $i({\bf{v}}_0,\xi)=\pm d({\bf{v}}_0,\xi)=0$ by Corollary~\ref{cor:deg_index}, which is again a contradiction.
\end{proof}

\begin{remark}
In \cite{BF75}, the authors constructed polyhedral embeddings with height functions with exactly three critical points for all closed oriented surfaces except the sphere. For the torus, such an embedding is given by Cs\'asz\'ar's torus: Figure~\ref{fig:csaszar} shows the middle critical point ${\bf{v}}_0$. As stated by Theorem~\ref{th:3cp}, $i({\bf{v}}_0,\bf{n})=-2$ and $d({\bf{v}}_0,\bf{n})=0$.

Note that the polyhedral setting is in complete contrast to the smooth case. There, only the sphere possesses an embedding into Euclidean three-space with a height function with exactly three critical points; there is no such embedding for a surface of higher genus. This question was suggested as an exercise by Hopf in \cite{ST38}.
\end{remark}

Let us now come back to our original goal of finding suitable $\xi$-isotopies of spherical polygons that will finally allow us to easily deduce their normal degrees.

\begin{proposition}\label{prop:xisotopy}
Let $W=(w_1,w_2,\ldots,w_n)$ be a simple spherical polygon and $\xi \in S^2$ admissible. We assume that all intersections of $W$ with $S(\xi)$ are vertices of $W$, such that the intersections of $W$ with the two hemispheres with pole $\pm\xi$ are disjoints unions of embedded polygonal arcs $W_i=(w_i,w_{i+1},\ldots,w_{i+s})$.

Then, there is a $\xi$-isotopy $W(t)$, $t\in [0,1]$, such that all the polygonal arcs of $W(1)$ in either of the two hemispheres are once-broken geodesics (that is, they each consist of exactly two great circle arcs).
\end{proposition}

\begin{proof}
By slightly perturbing the vertices of $W$ on $S(\xi)$ if necessary, we assume that no two of them are antipodal. We prove the statement by induction on the number of arcs that are not once-broken geodesics.

Let $R_i$ denote the region bounded by the polygonal arc $W_i$ in the respective hemisphere. $R_i$ shall be bounded by the minor great circle arc $w_iw_{i+s}$. If $W_i$ is not a once-broken geodesic but all the polygonal arcs of $W(t)$ lying in the region $R_i$ are once-broken geodesics, then we may deform these arcs, keeping their endpoints fixed, until there are no points of $W(t+\varepsilon)$ in $R_i \backslash T(W_i)$. Here, $T(W_i)$ is the spherical triangle with edge $w_iw_{i+s}$ which is not entered by any point of the $\xi$-isotopy $W_i(\cdot)$ constructed in Lemma~\ref{lem:deformation}. Then, we can apply this lemma to find a $\xi$-isotopy of $W_i$ to a once-broken geodesic. This $\xi$-isotopy extends to one from $W(t+\varepsilon)$ to $W(t+2\varepsilon)$ with one more polygonal arc being a once-broken geodesic.
\end{proof}

\begin{corollary}\label{cor:normalform}
Each simple spherical polygon $W$ is (up to subdivision and deletion of vertices) $\xi$-isotopic to a polygon $Y=Y(W)=(y_1,y^*_1,y_2,y^*_2,y_3,\ldots,y^*_{2r-1},y_{2r},y^*_{2r})$, where the vertices $y_1,y_2,\ldots,y_{2r}$ are points of a regular polygon on $S(\xi)$, the vertices $y^*_{2i+1}$ lie in the hemisphere with pole $\xi$, and the points $y^*_{2i}$ lie in the hemisphere with pole $-\xi$.
\end{corollary}

The points $y_1,y_2,\ldots,y_{2r}$ of $Y(W)$ in Corollary~\ref{cor:normalform} have now two cyclic orders: One is determined by their order on the oriented polygon $Y$, and the other is determined by the cyclic order of the vertices on $S(\xi)$, where the equator is oriented counterclockwise around the vector $\xi$. In this way, each general spherical polygon $W$ of non-positive critical point index $i(W,\xi)=2-2r$ determines a permutation $\Pi \in S_{2r}$ of the numbers from 1 to $2r$. Here $\Pi(k)-\Pi(j) \equiv l \textnormal{ mod } 2r$ means that in the cyclic order on $S(\xi)$, the vertex $y_k$ is $l-1$ vertices apart from $y_j$.

\begin{definition}
A once-broken geodesic $(y_{2i+1},y^*_{2i+1},y_{2i+2})$ of $Y(W)$ is called a \textit{chord} $[y_{2i+1},y_{2i+2}]$. A chord is \textit{free}, if its endpoints are adjacent vertices on $S(\xi)$. If $r \geq 2$, a free chord $[y_{2i+1},y_{2i+2}]$ is called \textit{positive} if $y_{2i+1}$ immediately precedes $y_{2i+2}$ in the order determined by $\Pi$, and  \textit{negative} if $y_{2i+1}$ immediately follows $y_{2i+2}$.
\end{definition}

\begin{lemma}\label{lem:freechords}
If $Y(W)$ has at least four vertices (i.e., if $r>1$), then it contains at least two free chords.
\end{lemma}

\begin{proof}
The chords $[y_{2i+1},y_{2i+2}]$, $i=1,\ldots,r$, are once-broken geodesics that all lie in the same hemisphere. Their endpoints lie on the equator and no two chords intersect each other. We say that they form a \textit{separating chord set}.

If $r=2$, then there are exactly two chords whose endpoints cannot be separated by them. Thus, the two chords are free.

If $r>2$, then any chord is either free itself, or it separates the equator into two arcs, each containing fewer than $r$ chords. Each of these smaller separating chord sets must have two free chords by induction, at least one of which is free for the larger chord set.
\end{proof}

\begin{remark}
The classification of $\xi$-isotopy classes of embedded polyhedral discs amounts to a classification of homotopy classes of embedded spherical polygons with fixed intersection with $S(\xi)$. Equivalently, we consider pairs of separating chord sets, each having the same number of chords. One chord set joins points of $S(\xi)$ in the upper hemisphere, the other set joins points of $S(\xi)$ in the lower hemisphere. All the chords taken together should determine a single closed curve on the sphere. The equivalence classes of such pairs of separating chord sets are in one-to-one correspondence with the $\xi$-isotopy classes of embedded polyhedral discs.

Separating chord sets are closely related to \textit{meanders}. A (closed) meander is a simple closed curve in the plane that intersects a straight line transversely in $2r$ points. Above, we considered simple closed curves on the sphere meeting $S(\xi)$ transversely in $2r$ points. So they correspond to meanders intersecting a circle. The number of meanders, either in the case of a straight line or a circle, is related to the Catalan numbers. There are combinatorial papers discussing their growth properties. The first algorithm that calculates the number of meanders was proposed by Lunnon in \cite{Lu69}. He dicussed the number of inequivalent ways of folding a map (or a strip of stamps), which corresponds to the number of meanders on a straight ray. The meanders on a straight line correspond to foldings of a closed strip of stamps. These correspondences and the number of meanders of multiple components were discussed by Di Francesco, Golinelli, and Guitter in \cite{FGG95}. However, the exact enumeration of closed meanders intersecting the line in $2r$ points is still not known. For their uniform sampling, Heitsch and Tetali suggested a Markov chain Monte Carlo approach in \cite{HT11}.

Meanders are also connected with configurations of RNA and protein folding \cite{HT11}. Moreover, they appear in the computation of the index of seaweed Lie algebras in the paper \cite{DK00} of Dergachev and Kirillov, and there is another reference to meanders in Arnold's work on ramified coverings of the four-dimensional sphere \cite{A88}.
\end{remark}

We now able describe an algorithm that determines the normal degree of a spherical polygon $W$ and an admissible $\xi \in S^2$:

Let $Y=Y(W)$ be one of the $\xi$-isotopic polygons of Corollary~\ref{cor:normalform}. The choice of a specific polygon $Y$ is not important as long as it induces the same permutation $\Pi$. In this sense, everything that follows is uniquely determined by $W$. By Lemma~\ref{lem:freechords}, we can choose a free chord in the hemisphere with pole $\xi$ and rename the vertices of $Y$ such that the vertices of the free chord are $y_1, y_2$. The polygon is then described by  a permutation $\Pi \in S_{2r}$, where $\Pi(k)=l$ means that $a_l$ is the $k$-th vertex on $S(\xi)$ beginning with $y_1$ going counterclockwise around $\xi$. So $\Pi(1)=1$ and either $\Pi(2)=2$ or $\Pi(2r)=1$ depending on whether $[y_1,y_2]$ is positive or negative. We then say that a (not necessarily free) chord $[y_{2i+1},y_{2i+2}]$ is positive or negative according to $\Pi(2i+1)$ coming before or after $\Pi(2i+2)$ in the cyclic order starting from $\Pi(1)=1$. In particular, if $[y_{2i+1},y_{2i+2}]$ is a free chord, then this notion of positivity and negativity coincides with the one given previously.

Let $N^+(\Pi)$ be the number of positive chords other than possibly $[y_{1},y_{2}]$, and let $N^-(\Pi)$ be the number of negative chords other than possibly $[y_{1},y_{2}]$. Clearly, $1+N^+(\Pi)+N_-(\Pi)$ equals the total number $r$ of chords in a hemisphere. Thus, \[-N^+(\Pi)-N^-(\Pi)=1-r=i(Y(W),\xi)=i(W,\xi).\]

\begin{theorem}\label{th:normaldegree}
For a spherical polygon $W$ and an admissible $\xi$ such that $i(W,\xi)=1-r<0$, let the permutation $(\Pi(1),\Pi(2),\ldots,\Pi(2r))$ represent its $\xi$-isotopy class. Then, the normal degree $d(W,\xi)$ equals \[d(W,\xi)=-N^+(\Pi)+N^-(\Pi).\]
In particular, this number is independent of the choice of the free chord with vertices $y_1$ and $y_2$.
\end{theorem}

\begin{proof}
In order to establish the theorem, we describe a deformation (which is not a $\xi$-isotopy) that suppresses the free chord $[y_{2i+1},y_{2i+2}]$. By this, we will increase the index and change the normal degree of $Y$ in a specific way.

Throughout this deformation $Y(t)$, $t \in [0,1]$, the vertices other than $y_{2i+1},y^*_{2i+1},y_{2i+2}$ remain fixed. First, $y_{2i+1}$ and $y_{2i+2}$ are moved linearly into the hemisphere with pole $-\xi$ in such a way that the polygon remains embedded. In particular, neither the index nor the normal degree change. During the next stage, $y^*_{2i+1}$ is moved linearly through $S(\xi)$ to the midpoint of the new arc $y_{2i+1}y_{2i+2}$.

As the point $y^*_{2i+1}$ passes through $S(\xi)$, the index $i(Y,\xi)$ changes to $i(Y(1),\xi)=i(Y,\xi)+1$ since $2r$ goes to $2r-2$. Moreover, the arc $y'_{2i+1}{y^*}'_{2i+1}$ of the polar polygon $Y'$ passes over $\xi$ or $-\xi$, and the normal degree changes by $\pm 1$. Specifically, if the chord $[y_{2i+1},y_{2i+2}]$ is positive, then $d(Y(1),\xi)=d(Y,\xi)+1$, and if the chord is negative, then $d(Y(1),\xi)=d(Y,\xi)-1$.

To see the last fact, let us investigate the movement of the arc $y'_{2i+1}{y^*}'_{2i+1}$. The normal degree only changes as this arc passes over $\pm \xi$. Thus, we can restrict our attention to the points $y_{2i+1},y^*_{2i+1},y_{2i+2}$ of $Y$. Without changing the index or the degree, we move $y_{2i+1}$ and $y_{2i+2}$ in such a way that their connecting arc has length $\pi/2$. They still lie on $S(\xi)$. Furthermore, we move $y^*_{2i+1}$ to the arc $a$ connecting the midpoint $m$ of $y_{2i+1}y_{2i+2}$ with $\xi$. If $y^*_{2i+1}$ approaches $\xi$,  $y'_{2i+1}$ approaches $-y_{2i+2}$ or $y_{2i+2}$ and ${y^*}'_{2i+1}$ approaches $-y_{2i+1}$ or $y_{2i+1}$ depending on the chord $[y_{2i+1},y_{2i+2}]$ being positive or negative, respectively.

\begin{figure}[!ht]
	\centerline{
		\begin{overpic}[height=0.3\textwidth]{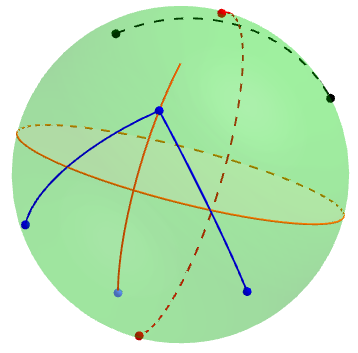}
			\cput(27,85){{$y'_{2i+1}$}}
			\cput(89,68){{${y^*}'_{2i+1}$}}
			\color{blue}
			\cput(5,31){{$y_{2i+1}$}}
			\cput(35,70){{$y^*_{2i+1}$}}
			\cput(78,15){{$y_{2i+2}$}}
			\cput(33,10){{$m$}}
			\color{red}
			\cput(30,25){{$a$}}
			\cput(58,95){{$\xi$}}
			\cput(42,0){{$-\xi$}}
			\cput(62,65){{$\gamma$}}
			\cput(73,39){{$S(\xi)$}}
		\end{overpic}
		\relax\\
		\begin{overpic}[height=0.3\textwidth]{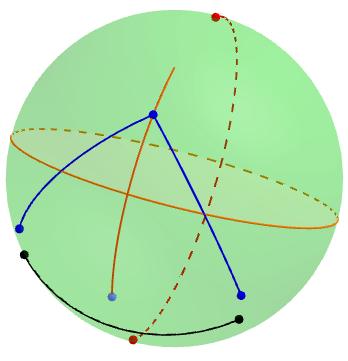}
			\cput(75,6){{${y^*}'_{2i+1}$}}
			\cput(1,25){{$y'_{2i+1}$}}
			\color{blue}
			\cput(15,35){{$y_{2i+2}$}}
			\cput(35,71){{$y^*_{2i+1}$}}
			\cput(77,17){{$y_{2i+1}$}}
			\cput(32,11){{$m$}}
			\color{red}
			\cput(28,25){{$a$}}
			\cput(57,95){{$\xi$}}
			\cput(40,0){{$-\xi$}}
			\cput(61,65){{$\gamma$}}
			\cput(74,39){{$S(\xi)$}}
		\end{overpic}
	}
	\caption{Change of the normal degree when $y^*_{2i+1}$ moves to the other side of $S(\xi)$}\label{fig:suppression}
\end{figure}

Let us focus on the case of a positive chord first. As $y^*_{2i+1}$ moves along $a$ toward $m$, $y'_{2i+1}$ moves on the arc connecting $-y_{2i+2}$ and $\xi$, and ${y^*}'_{2i+1}$ moves on the arc connecting $-y_{2i+1}$ and $\xi$ toward $\xi$. Instead of moving $y_{2i+1}$ and $y_{2i+2}$ down to the hemisphere with pole $-\xi$ and then moving $y^*_{2i+1}$ to $m$, we move $\xi$ linearly to $-m$. To investigate the change of the normal degree $d(Y(t),\xi)$, we consider the oriented great circle arc $\gamma$ going from $\xi$ to $-\xi$ passing over $-m$ (cf. Section~\ref{sec:degree_index}). At the beginning, the arc $y'_{2i+1}{y^*}'_{2i+1}$ crosses $\gamma$ from left to right and contributes with $-1$ to the normal degree. When $\xi$ went through $y'_{2i+1}{y^*}'_{2i+1}$, there is no intersection and hence no contribution anymore. This implies that the normal degree increases by one.

We can apply a similar reasoning in the case that $[y_{2i+1},y_{2i+2}]$ is negative. Now, as $y^*_{2i+1}$ moves along $a$ toward $m$, $y'_{2i+1}$ moves on the arc connecting $y_{2i+2}$ and $-\xi$, and ${y^*}'_{2i+1}$ moves on the arc connecting $y_{2i+1}$ and $-\xi$ toward $-\xi$. When $\xi$ is moved linearly to $-m$, the arc $y'_{2i+1}{y^*}'_{2i+1}$ first does not cross $\gamma$ and eventually crosses it from left to right as $-\xi$ passes over that arc. Thus, the normal degree decreases by one.

In summary, we start with the polygon $Y$ and move one free chord $[y_{2i+1},y_{2i+2}]$ other than $[y_{1},y_{2}]$ to the other side of the hemisphere. Depending on the chord being positive or negative, we increase or decrease the normal degree, respectively. The index always increases by one. After this step, we apply Lemma~\ref{lem:deformation} to deform $Y(1)$ in a $\xi$-isotopic way such that the new polygon only has once-broken geodesics. All free chords other than $[y_{2i+1},y_{2i+2}]$ remain free and stay positive or negative. We continue this procedure until we have no free chords other than $[y_1,y_2]$ left. By Lemma~\ref{lem:freechords}, this is first the case if $r=1$, meaning that the index is zero. By Corollary~\ref{cor:index0}, the normal degree is then zero as well. It follows that $d(W,\xi)+N^+(\Pi)-N^-(\Pi)=0$.
\end{proof}

Note that the index and the normal degree are not sufficient to determine the $\xi$-isotopy class of the embedded polygon if $r>3$. However, we have the following relationships.

\begin{proposition}\label{prop:degree_existence}
Let $W$ be a simple spherical polygon in general position. Then, $\vert d(W,\xi)\vert \leq \vert i(W,\xi)\vert $ and $d(W,\xi)-i(W,\xi)$ is even for any admissible $\xi \in S^2$. Conversely, given any two integers $i,d$ with $i<1$, $\vert i\vert \geq \vert d\vert $, and $d-i$ even, there exists a simple spherical polygon $W$ and an appropriate choice of $\xi$ such that $i(W,\xi)=i$ and $d(W,\xi)=d$.
\end{proposition}

\begin{proof}
If $i(W,\xi)=1$, then $d(W,\xi)=\pm 1$ by Proposition~\ref{prop:index1}. If $i(W,\xi)=0$, we can apply Corollary~\ref{cor:index0} to deduce $d(W,\xi)=0$. Otherwise, $i(W,\xi)<0$ and Theorem~\ref{th:normaldegree} gives $i(W,\xi)=-N^+(\Pi)-N^-(\Pi)$ and $d(W,\xi)=-N^+(\Pi)+N^-(\Pi)$ for non-negative integers $N^+(\Pi)$ and $N^-(\Pi)$. In particular, we get that $\vert d(W,\xi)\vert \leq \vert i(W,\xi)\vert $ and that $d(W,\xi)-i(W,\xi)=2N^-(\Pi)$ is even (which we already know from Corollary~\ref{cor:classes}).

Conversely, any spherical polygon $W$ that is $\xi$-isotopic to a great circle satisfies $i(W,\xi)=d(W,\xi)=0$. So let $i<0$ and $r=1-i>1$. For any $k$ satisfying $1\leq k\leq r-1$, let us consider the permutation $(\Pi(1),\Pi(2),\ldots,\Pi(2r))=(1,2,\ldots,2k-1,2k,2r,2r-2,\ldots,2k+1,2k)$, for $k=0$ let us consider $(\Pi(1),\Pi(2),\ldots,\Pi(2r))=(2r,2r-1,\ldots,2,1)$. We can easily find a polygon $W$ and a normal $\xi$ such that their corresponding permutation is $\Pi$. It has $N^+(\Pi)=k$ positive and $N^-(\Pi)=r-k-1$ negative chords other than $(1,2)$ or $(2,1)$. Hence, $i(W,\xi)=1-r=i$ and $d(W,\xi)=-N^+(\Pi)+N^-(\Pi)=r-2k-1$ for any choice of $k$ satisfying $0\leq k\leq r-1$. Thus, we get any $d(W,\xi)$ between $r-1=-\vert i \vert$ and $r-2(r-1)-1=1-r=\vert i \vert$ such that $d(W,\xi)-i$ is even.
\end{proof}


\section{Possible shapes of Gauss images}\label{sec:shape_analysis}

In this section, we describe how Theorem~\ref{th:egregium2} limits the possible shapes of Gauss images of polyhedral vertex stars. We start in Section~\ref{sec:shape_general} with a general formula relating the number of positive and negative components in the Gauss image with the number of inflection faces in the vertex star. Then, we restrict ourselves to embedded polyhedral vertex stars in Section~\ref{sec:shape_embedded}. We deduce from Proposition~\ref{prop:degree_existence} that there is at most one positive component in the Gauss image. This component equals the Gauss image of the convex hull of the vertex star. Finally, we discuss the case of Gauss images without self-intersections in Section~\ref{sec:shape_simple}. It turns out that the only possible shapes are convex spherical polygons if the discrete Gaussian curvature is positive and spherical pseudo-quadrilaterals if the discrete Gaussian curvature is negative. We conclude with the case when non-convex faces are present.


\subsection{Gauss images of general polyhedral vertex stars}\label{sec:shape_general}

We have seen in Section~\ref{sec:degree} that it is more convenient to work with spherical polygons and their polars instead of polyhedral vertex stars. So in the following, let $W$ be a spherical polygon in general position and $W'$ its polar. Note that inflection faces in the vertex star correspond to \textit{inflection edges} in the spherical polygon. These are edges where its two neighboring edges lie in different hemispheres defined by the great circle through the edge.

The complement of $W'$ consists of several connected components on each of which the winding number is constant. We call a set of such components \textit{connected} if any two points can be connected by a continuous path on $S^2$ that does not pass through any vertex of $W'$. This means that two components that share a non-trivial part of an edge of $W'$ are considered to be connected. Connected components of such sets are defined appropriately.

\begin{definition}
We denote by $C_-^k$ the number of connected components in the set $\left\{\xi \in S^2 \vert w(W',\xi)\leq -k\right\}$ and we denote by $C_+^k$ the number of connected components in the set $\left\{\xi \in S^2 \vert w(W',\xi)\geq k\right\}$. \[C_-:=\sum\limits_{k=1}^\infty C_-^k \textnormal{ and } C_+:=\sum\limits_{k=1}^\infty C_+^k \] are the number of \textit{negative} and \textit{positive layers} of $W'$, respectively.
\end{definition}

\begin{remark}
The reason why we call $C_-$ (or $C_+$) the number of negative (or positive) layers is explained as follows. We split a component of winding number $\pm k$ into $k$ components with a plus or minus sign. Now, we let components with the same sign connected to each other merge to one layer such that at no point of time, a point of the sphere is covered by more than one component of the layer. Then, $k$ layers of sign $\pm$ overlap a point with winding number $\pm k$ and $C_\pm$ denotes the total number of layers with a plus or minus sign on the sphere.
\end{remark}

\begin{theorem}\label{th:shape}
Let $I$ denote the number of inflection edges of a spherical polygon $W$ in general position, and let $c$ be zero or one if the number of self-intersections of $W$ is even or odd, respectively. Then, $I$ corresponds to the number of right turns of $W'$, and \[I+2C_+-2C_-=2+2c.\]
\end{theorem}
\begin{proof}
The idea of proof is to compute the algebraic area $A(W')$ of $W'$ in two different ways: Once by considering the layers on the sphere defined by $W'$ and the other time by using Theorem~\ref{th:egregium2}.

Let us consider the layers $L$ defined by $W'$ described in the previous remark. By $\sigma(L)$ we denote the sign of $L$. Then, $A(W')=\sum_L \sigma(L) A(L)$. Since the winding number decreases by one if $W'$ is crossed from left to right, each layer is a simple spherical polygon whose $n(L)$ vertices are among the vertices and the intersection points of $W'$ and whose edges lie on the edges of $W'$. If its interior angles are denoted by $\alpha'_j(L)$, then its area $A(L)$ is given by \[\sum_{j=1}^{n(L)} \alpha'_j-(n(L)-2)\pi=\sum_{j=1}^{n(L)} (\alpha'_j(L)-\pi)+2\pi.\]

It follows that \[A(W')=\sum_L\sum_{j=1}^{n(L)} \sigma(L)(\alpha'_j(L)-\pi)+(C_+-C_-)2\pi.\]

We now sum up the contributions $\sigma(L)(\alpha'_j(L)-\pi)$ around each vertex and each intersection point of $W'$ separately. Let us first consider the case of an intersection point. Regarding the winding numbers around that point, we distinguish three cases: Either the winding numbers are $\pm 1$ and zero, or they are all non-negative, or they are all non-positive, see Figure~\ref{fig:cases_winding} with $k \geq 0$.

\begin{figure}[!ht]
	\centerline{
		\begin{overpic}[height=0.3\textwidth]{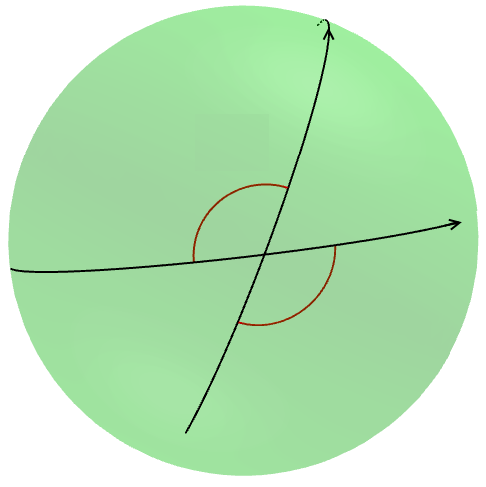}
			\cput(25,70){{$+1$}}
			\cput(25,28){{$0$}}
			\cput(75,70){{$0$}}
			\cput(75,28){{$-1$}}
			\color{red}
			\cput(59,40){{$\alpha'_-$}}
			\cput(48,52){{$\alpha'_+$}}
		\end{overpic}
		\relax\\
		\begin{overpic}[height=0.3\textwidth]{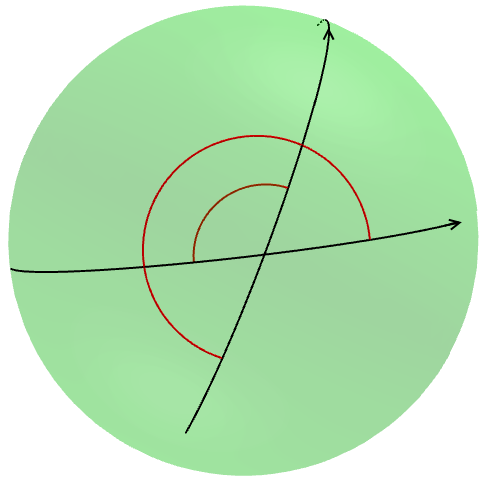}
		  \cput(25,70){{$k+2$}}
			\cput(25,28){{$k+1$}}
			\cput(75,70){{$k+1$}}
			\cput(75,28){{$k$}}
			\color{red}
			\cput(39,60){{$\alpha'_2$}}
			\cput(48,52){{$\alpha'_1$}}
		\end{overpic}
		\relax\\
		\begin{overpic}[height=0.3\textwidth]{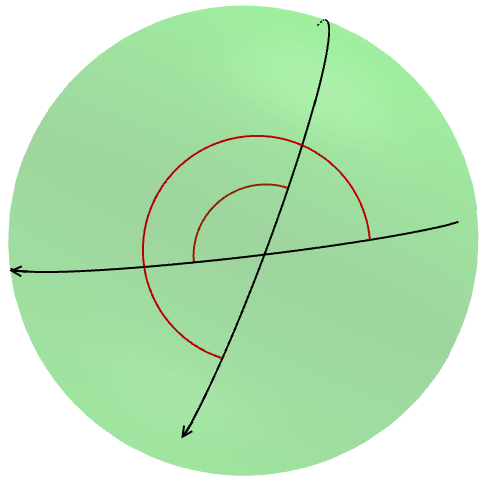}
		  \cput(25,70){{$-k-2$}}
			\cput(25,28){{$-k-1$}}
			\cput(75,70){{$-k-1$}}
			\cput(75,28){{$-k$}}
			\color{red}
			\cput(39,60){{$\alpha'_2$}}
			\cput(48,52){{$\alpha'_1$}}
		\end{overpic}
	}
	\caption{Three cases for the winding numbers around an intersection point of $W'$}\label{fig:cases_winding}
\end{figure}

In the first case, a positive layer and a negative layer meet at the intersection point. In the sum of angles, we once consider $(\alpha'_+-\pi)$ and once $-(\alpha'_--\pi)$. Due to $\alpha'_+=\alpha'_-$, the contributions sum up to zero.

In the second case, the top layer of $\left\{\xi \in S^2 \vert w(W',\xi)\geq k\right\}$ contains the intersection point in its interior. The angles corresponding to the top layers of $\left\{\xi \in S^2 \vert w(W',\xi)\geq k+1\right\}$ and $\left\{\xi \in S^2 \vert w(W',\xi)\geq k+2\right\}$ are given by $\alpha'_1$ and $\alpha'_2$ as depicted in the figure. So the contribution to the sum of angles around the intersection point is $(\alpha'_1-\pi)+(\alpha'_2-\pi)=0$ due to $\alpha'_1+\alpha'_2=2\pi$. By a similar argument, the corresponding contribution in the third case is zero as well.

We now come to the contribution around a vertex $w_i'$ of $W'$, distinguishing the cases whether the edge $w_iw_{i+1}$ is an inflection edge or not. It follows from our observations in Lemma~\ref{lem:angle} that if $w_iw_{i+1}$ is an inflection edge or not, then $W'$ makes a left or right turn at $w'_i$, respectively. Therefore, the windings number are as depicted in Figure~\ref{fig:vertex_winding}.

\begin{figure}[!ht]
	\centerline{
		\begin{overpic}[height=0.3\textwidth]{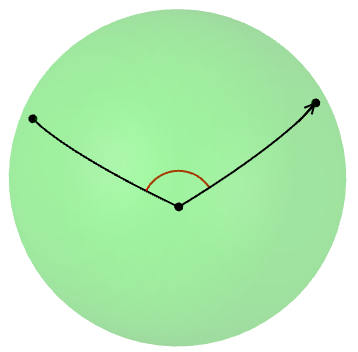}
			\cput(51,35){{$w'_{i}$}}
			\cput(51,55){{$k$}}
			\cput(51,22){{$k-1$}}
			\color{red}
			\cput(51,44){{$\alpha'$}}
		\end{overpic}
		\relax\\
		\begin{overpic}[height=0.3\textwidth]{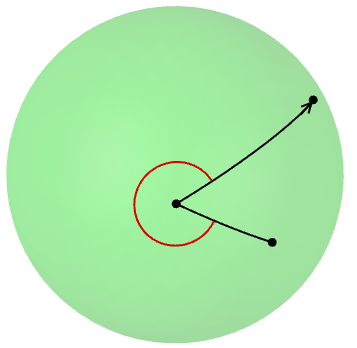}
			\cput(59,39){{$w'_{i}$}}
			\cput(33,39){{$k$}}
			\cput(77,39){{$k-1$}}
			\color{red}
			\cput(46,39){{$\alpha'$}}
		\end{overpic}
	}
	\caption{Two cases for the winding numbers around a vertex of $W'$}\label{fig:vertex_winding}
\end{figure}

If $k > 0$, then we have exactly one positive layer with boundary vertex $w'_i$. Its contribution to the angle sum is $\alpha'-\pi$. In the case that $k \leq 0$, then we have exactly one negative layer with boundary vertex $w'_i$. Its contribution to the angle sum is again $-(2\pi-(\alpha'-\pi))=\alpha'-\pi$. By Lemma~\ref{lem:angle}, $\alpha'-\pi=-\textnormal{length}(w_iw_{i+1})$ if $w_iw_{i+1}$ is not an inflection edge and $\alpha'-\pi=\pi-\textnormal{length}(w_iw_{i+1})$ if it is.

Summing up the contributions over all vertices, we thus obtain $A(W')=(C_+-C_-)2\pi+I\pi-\textnormal{length}(W)$. On the other hand,  by Theorem~\ref{th:egregium2} we have $A(W')=2\pi(1+c)-\textnormal{length}(W)$. Therefore, \[(C_+-C_-)2\pi+I\pi-\textnormal{length}(W)=A(W')=2\pi(1+c)-\textnormal{length}(W),\] from which $I+2C_+-2C_-=2+2c$ follows.
\end{proof}

Since the polar of the polar polygon is the polygon itself, we cannot expect any limitation of shapes of polar polygons and hence Gauss images in the general case. However, Theorem~\ref{th:shape} leads to certain restrictions if we for example know that $W$ is simple or $W'$ is in addition simple. We discuss these cases in the following Sections~\ref{sec:shape_embedded} and~\ref{sec:shape_simple}.

\begin{remark}
Theorem~\ref{th:shape} allows us to compute the winding numbers directly from the shape of the polar polygon without constructing a regular homotopy of $W$ to a polygon for which the winding numbers are known. Indeed, the number of right turns $I$ and the mod 2 intersection number $c$ can be calculated immediately from $W'$. This fixes $C_+-C_-$. Finding an initial set of winding numbers is done using their property that the winding number decreases by one if $W'$ is crossed from left to right. Now, adding $k$ to all winding numbers increases $C_+$ by $k$ and decreases $C_-$ by $k$, changing $C_+-C_-$ by $2k$. This corresponds to adding $4k\pi$ to the (a priori) undefined algebraic area of $W'$. Knowing $C_+-C_-$, we can easily adjust the winding numbers such that $I+2C_+-2C_-=2+2c$.
\end{remark}


\subsection{Gauss images of embedded polyhedral vertex stars}\label{sec:shape_embedded}

We are now restricting to the case of an embedded polyhedral vertex star, meaning that the corresponding spherical polygon $W$ is simple. By Theorem~\ref{th:shape}, we know that $I+2C_+-2C_-=2$, where $I$ is the number of inflection faces and also the number of right turns in the Gauss image. A significant restriction of possible shapes of Gauss images comes from the observation that the Gauss image contains at most one positive layer:

\begin{proposition}\label{prop:positive_layer}
Let $W$ be a simple spherical polygon associated to an embedded polyhedral vertex star. Then, the number $C_+$ of positive layers of the polar polygon $W'$ is at most one. If a positive layer exists, then it is the Gauss image of the boundary of the convex hull of the vertex star, in particular it is convex and contains no right turns.
\end{proposition}
\begin{proof}
Let $\xi$ be general for $W$. From Proposition~\ref{prop:degree_existence} we know that \[\vert w(W',\xi)+w(W',-\xi)\vert=\vert i(W,\xi)\vert\geq\vert d(W,\xi)\vert=\vert w(W',\xi)-w(W',-\xi)\vert.\] It follows that $w(W',\xi)$ and $w(W',-\xi)$ cannot have a different sign. So if $w(W',\xi)$ is positive, then $w(W',\xi)$ is nonnegative, such that $i(W,\xi)\geq 1$. On the other hand, $i(W,\xi)\leq 1$ due to Proposition~\ref{prop:index_equality}. It follows that if $w(W',\xi)$ is positive, then $w(W',\xi)=1$ and $w(W',-\xi)=0$. Furthermore, $w(W',\pm \xi)=1$ if and only if $i(W,\xi)=1$. Hence, the positive layers correspond to these spherical regions on which the index equals one.

By definition, $i(W,\xi)=1$ if and only if $\langle \xi, w_k \rangle$ has the same sign for all vertices $w_k$, say all scalar products are negative. This is the case if and only if $\langle \xi, w \rangle <0$ for all points $w$ of the convex hull of the rays starting in the origin and passing through $W$. Equivalently, we can take the convex hull of the vertex star. It is easy to see that the convex hull either consists of the whole Euclidean space or is a convex cone. Only in the latter case points $\xi \in S^2$ with $i(W,\xi)=1$ exist. Then, the spherical region on which $i(W,\xi)=1$ is the Gauss image of the cone surface. The cone surface is itself a polyhedral vertex star. In particular, there is at most one positive layer and that layer is a convex spherical polygon.
\end{proof}

In combination with Proposition~\ref{prop:positive_layer}, Theorem~\ref{th:shape} significantly limits the shapes of Gauss images of embedded vertex stars. In the case that there is no plane through the vertex such that the whole vertex star lies in one half-space, then the Gauss image contains only negatively signed components and $I=2C_-+2$. The number of right turns $I$ corresponds now to the number of \textit{corners} in the Gauss image, i.e., to vertices whose interior angle is less than $\pi$. If there exists such a plane, then the Gauss image contains exactly one positively counting component that is convex and $I=2C_-$. The number of right turns $I$ corresponds to the number of corners in the negatively signed components of the Gauss image.

Let us consider an embedded polyhedral surface $P$ and one of its vertices ${{\bf{v}}}$. We repeat the notion of the positive and negative parts of discrete Gaussian curvature defined in the paper \cite{BK82} of Brehm and K\"uhnel. According to Wolfgang K\"uhnel, this notion actually goes back to Yuri Burago.

\begin{definition}
Let $C$ be the convex hull of the star of ${{\bf{v}}}$. If $C$ is the whole Euclidean space, we define the \textit{positive part of the discrete Gaussian curvature} $K_+({{\bf{v}}}):=0$. Otherwise let $K_+({{\bf{v}}})$ be the discrete Gaussian curvature of the surface of $C$ at ${{\bf{v}}}$. The \textit{negative part of the discrete Gaussian curvature} is given by $K({{\bf{v}}}):=K({{\bf{v}}})-K_+({{\bf{v}}})$.
\end{definition}

As a corollary of Proposition~\ref{prop:positive_layer} and Theorem~\ref{th:egregium}, we obtain a natural interpretation of the positive and negative part of the curvature in terms of the areas of the positive and negative layers of the Gauss image.

\begin{corollary}\label{cor:parts_interpretation}
\[K_+({{\bf{v}}})=\int\limits_{\xi \in S^2: w(g({{\bf{v}}}),\xi)>0} w(g({{\bf{v}}}),\xi) d\xi \textnormal{ and } K_-({{\bf{v}}})=\int\limits_{\xi \in S^2: w(g({{\bf{v}}}),\xi)<0} \lvert w(g({{\bf{v}}}),\xi) \rvert d\xi.\]
\end{corollary}

\begin{remark}
Brehm and K\"uhnel showed in \cite{BK82} that, for any compact polyhedral surface $P$ without boundary, there is a sequence of smooth surfaces with the following properties: Each surface is homeomorphic to $P$, the sequence is converging to $P$ with respect to the Hausdorff metric, and in any open subset of $P$ whose boundary does not contain any vertex, the discrete integrals of Gaussian curvature, absolute curvature, and mean curvature are converging to their discrete counterparts. One might assume that our Theorem~\ref{th:egregium} follows from their considerations. But we cannot deduce straightaway from their construction that the Gauss images of a vertex neighborhood are converging to the corresponding discrete Gauss image.
\end{remark}


\subsection{Gauss images without self-intersections and the presence of non-convex faces}\label{sec:shape_simple}

We are now investigating the star around a vertex ${{\bf{v}}}$ of an embedded polyhedral surface $P$ and want to characterize these Gauss images that are free of self-intersections. This means that either $C_+=1$ and $C_-=0$ or $C_+=0$ and $C_-=1$. By Theorem~\ref{th:shape}, $I=0$ in the first and $I=4$ in the latter case. This means that if $K({{\bf{v}}})>0$, $g({{\bf{v}}})$ is a convex spherical polygon and no inflection faces are present in the vertex star, meaning that the vertex star bounds a convex cone as in Figure~\ref{fig:convex}. If $K({{\bf{v}}})<0$, $g({{\bf{v}}})$ has exactly four corners, i.e., it is a spherical pseudo-quadrilateral. Since the vertex star contains exactly four inflection faces, it is saddle-shaped as in Figure~\ref{fig:Gaussian}.

So far, we only considered simplicial polyhedral surfaces $P$. Since we mainly investigated the geometry of vertex stars and their Gauss images, it would have made no difference at all if we had allowed faces with more than just three vertices as long as they had a convex angle at the vertex under consideration. We now discuss the changes that appear when non-convex faces with a reflex angle at the vertex of the star are present.

Our discussion in Sections~\ref{sec:Gauss} and Sections~\ref{sec:degree} does not significantly change. However, we note that we used convexity of faces in the definition of the critical point index. Using the formulation as the middle vertex index, \[i({{\bf{v}}},\xi)=1-\frac{1}{2}M({{\bf{v}}},\xi),\] where $M({{\bf{v}}},\xi)$ was the number of faces $f \sim {{\bf{v}}}$ such that  $\langle \xi, \cdot \rangle$ does neither attain its maximum or its minimum on $f$ at ${{\bf{v}}}$. Faces with a reflex angle now count twice in the sum $M({{\bf{v}}},\xi)$.

In Section~\ref{sec:shape_analysis}, Theorem~\ref{th:shape} needs an adaptation due to the following amendment to Lemma~\ref{lem:angle}:

\begin{lemma}\label{lem:angle2}
Let $f_1, f_2, f_3$ be three consecutive faces of the star of the vertex ${{\bf{v}}}$, ordered in counterclockwise direction around ${{\bf{v}}}$ with respect to the given orientation of $P$. Let $\alpha:=\alpha_{2}({{\bf{v}}})>\pi$ be the angle of the non-convex face $f_2$ at ${{\bf{v}}}$ and $\alpha'$ the spherical angle $\angle {\bf{n}}_{3}{\bf{n}}_{2}{\bf{n}}_{1}$.
\begin{enumerate}
\item If $f_2$ is not an inflection face, then $\alpha'=3\pi-\alpha$.
\item If $f_2$ is an inflection face, then $\alpha'=2\pi-\alpha$.
\end{enumerate}
\end{lemma}
\begin{proof}

(i) Consider the triangle $f'_2$ spanned by the two edges of $f_2$ incident to ${{\bf{v}}}$ and replace $f_2$ by $f'_2$. Then, we are in the case of Lemma~\ref{lem:angle}~(i) and ${\bf{n}}_{f'_2}=-{\bf{n}}_{2}$. It follows that the spherical angle $\angle {\bf{n}}_{3}(-{\bf{n}}_{2}){\bf{n}}_{1}$ equals $\pi-(2\pi-\alpha)=\alpha-\pi$. Thus, \[\alpha'=\angle {\bf{n}}_{3}{\bf{n}}_{2}{\bf{n}}_{1}=2\pi-\angle {\bf{n}}_{1}{\bf{n}}_{2}{\bf{n}}_{3}=2\pi-\angle {\bf{n}}_{3}(-{\bf{n}}_{2}){\bf{n}}_{1}=3\pi-\alpha.\]

\begin{figure}[!ht]
	\centerline{
		\begin{overpic}[height=0.3\textwidth]{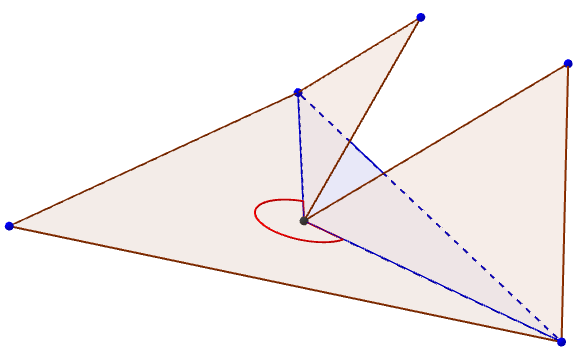}
			\put(58,45){\contour{white}{$f_1$}}
			\put(33,26){\contour{white}{$f_2$}}
			\put(82,30){\contour{white}{$f_3$}}
			\put(59,32){\contour{white}{$f_2'$}}
			\cput(57,22){{${{\bf{v}}}$}}
			\color{red}
			\cput(50,22){{$\alpha$}}
		\end{overpic}
		\relax\\
		\begin{overpic}[height=0.3\textwidth]{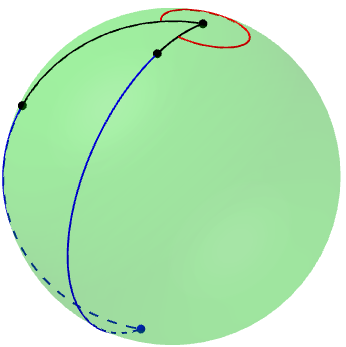}
			\cput(9,63){{${\bf{n}}_{1}$}}
			\cput(46,89){{${\bf{n}}_{2}$}}
			\cput(49,78){{${\bf{n}}_{3}$}}
			\cput(43,0){{$\mbox{-}{\bf{n}}_{2}$}}
			\color{red}
			\cput(62,88){{$\alpha'$}}
		\end{overpic}
	}
	\caption{Lemma~\ref{lem:angle2} (i): $\alpha'=3\pi-\alpha$ if $f_2$ is not an inflection face}\label{fig:angle3}
\end{figure}

(ii) With a similar argument as in (i), we can use the result of Lemma~\ref{lem:angle}~(ii) to obtain that the spherical angle $\angle {\bf{n}}_{3}(-{\bf{n}}_{2}){\bf{n}}_{1}$ equals $2\pi-(2\pi-\alpha)=\alpha$, such that \[\alpha'=\angle {\bf{n}}_{3}{\bf{n}}_{2}{\bf{n}}_{1}=2\pi-\angle {\bf{n}}_{3}(-{\bf{n}}_{2}){\bf{n}}_{1}=2\pi-\alpha.\qedhere\]

\begin{figure}[!ht]
	\centerline{
		\begin{overpic}[height=0.3\textwidth]{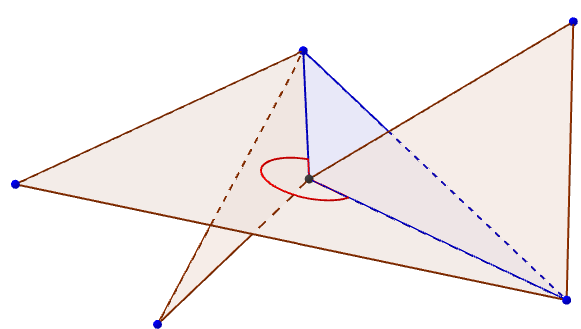}
			\put(33,11){\contour{white}{$f_1$}}
			\put(33,30){\contour{white}{$f_2$}}
			\put(82,30){\contour{white}{$f_3$}}
			\put(55,35){\contour{white}{$f_2'$}}
			\cput(57,25){{${{\bf{v}}}$}}
			\color{red}
			\cput(49,26){{$\alpha$}}
		\end{overpic}
		\relax\\
		\begin{overpic}[height=0.3\textwidth]{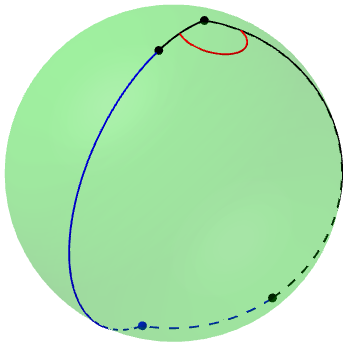}
			\cput(79,8){{${\bf{n}}_{1}$}}
			\cput(57,97){{${\bf{n}}_{2}$}}
			\cput(49,79){{${\bf{n}}_{3}$}}
			\cput(43,0){{$\mbox{-}{\bf{n}}_{2}$}}
			\color{red}
			\cput(62,87){{$\alpha'$}}
		\end{overpic}
	}
	\caption{Lemma~\ref{lem:angle2} (ii): $\alpha'=2\pi-\alpha$ if $f_2$ is an inflection face}\label{fig:angle4}
\end{figure}

\end{proof}

Let us compare the new expressions for the spherical angles with the previous ones. We then observe that in the proof of Theorem~\ref{th:shape}, non-convex inflection faces behave the same as usual inflection faces, but if a non-convex face with a reflex angle at ${{\bf{v}}}$ is not an inflection face, it actually counts as two inflection faces (in some sense canceling each other). We still obtain the formula \[I+2C_+-2C_-=2+2c\] for the Gauss image, but $I$ is now the number of inflection faces where faces with a reflex angle that are not inflection faces count twice. $I$ also does not correspond to the number of left turns in the Gauss image anymore, but to the number of left turns at normals to faces with a convex angle plus the number of right turns at normals to faces with a reflex angle plus twice the number of left turns at normals to faces with a reflex angle.

Whereas we have no changes in Section~\ref{sec:shape_embedded}, we observe additional cases of embedded polyhedral vertex stars whose Gauss image has no self-intersections if the discrete Gaussian curvature is negative. These cases are depicted in Figure~\ref{fig:additional}. In the case that the vertex star contains one reflex angle, the corresponding face might be an inflection face or not. In the first case, due to $I=4$ three more inflection faces are present and the Gauss image is a spherical pseudo-triangle. Its corners are the normals of the three inflection faces that do not have a reflex angle at the corner. In the second case, the Gauss image is again a spherical pseudo-triangle whose corners are the normals of the face with a reflex angle and of the two inflection faces. Since each simple spherical polygon with clockwise orientation has at least two left turns, there cannot be more than two reflex angles in the vertex star. In the case of two reflex angles, both faces with a reflex angle have to be inflection faces. So there are two more inflection faces whose normals are the corners of the Gauss image, which is now a spherical pseudo-digon. The faces with a reflex angle are inflection faces because of the following observation: If we go along the star of ${{\bf{v}}}$ from one reflex angle to the other, there has to be at least one inflection face in between.

Let us quickly discuss why this is true. We split the two faces $f$ with a reflex angle into two faces $f_1$, $f_2$ that have an angle of less than $\pi$ at ${{\bf{v}}}$ and suppose that $f_1,f_2,f'_1,f'_2$ appear in counterclockwise order around the vertex star. Let us consider the faces in the star of ${{\bf{v}}}$ following $f_2$. Until we encounter the first inflection face or $f'_2$, the faces form part of the boundary of a convex cone. Suppose that there is no inflection face between $f_2$ and $f'_1$. Then, the faces $f_2$ and the ones following it form part of the boundary of a convex cone. If not all of these faces lie in the same of the two half-spaces determined by $f$, one face would intersect $f$, contradicting that the vertex star is embedded. But $f'$ intersects the plane through $f$, so $f'_1$ and $f'_2$ cannot lie in a common half-space. Assuming the absence of an inflection face between $f$ and $f'$ results in an intersection of these two faces. 

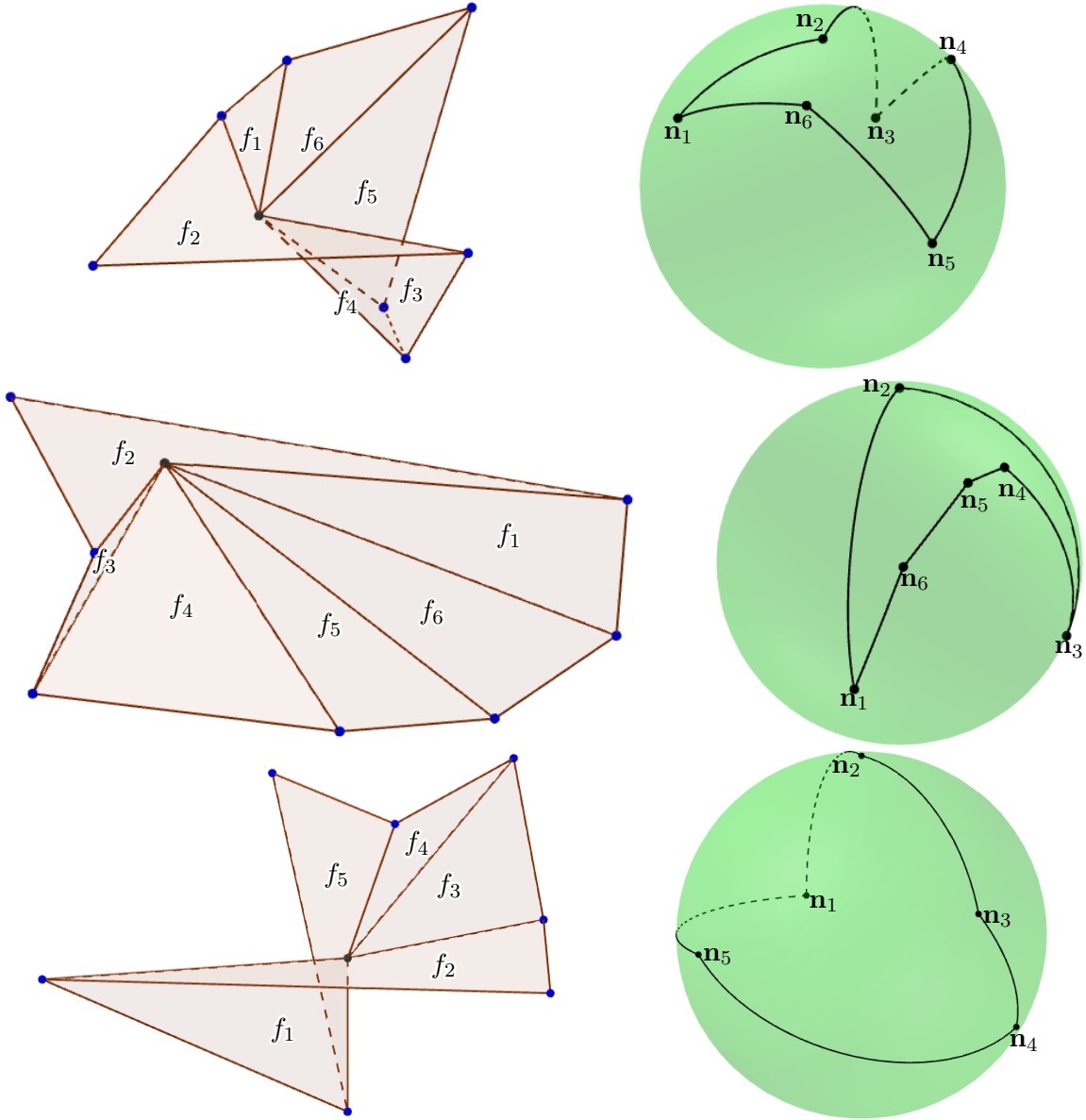
\begin{figure}[!ht]
	\centerline{
		\begin{overpic}[height=0.3\textwidth]{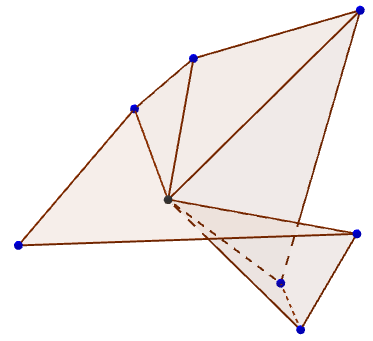}
			\put(40,55){\contour{white}{$f_1$}}
			\put(25,32){\contour{white}{$f_2$}}
			\put(79,18){\contour{white}{$f_3$}}
			\put(63,15){\contour{white}{$f_4$}}
			\put(68,42){\contour{white}{$f_5$}}
			\put(55,55){\contour{white}{$f_6$}}	
		\end{overpic}
		\relax\\
		\begin{overpic}[height=0.3\textwidth]{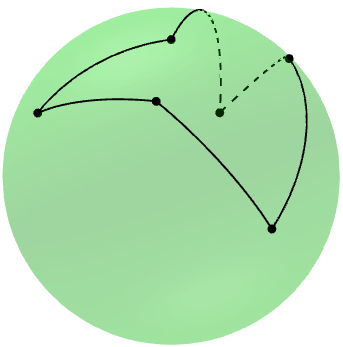}
				\cput(11,62){{$\nw_{1}$}}
				\cput(46,92){{$\nw_{2}$}}
				\cput(65,62){{$\nw_{3}$}}
				\cput(84,86){{$\nw_{4}$}}
				\cput(81,28){{$\nw_{5}$}}
				\cput(43,66){{$\nw_{6}$}}
		\end{overpic}
	}
	\centerline{
		\begin{overpic}[height=0.3\textwidth]{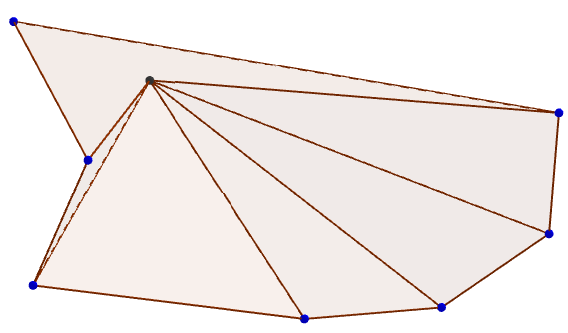}
		  \put(78,32){\contour{white}{$f_1$}}
			\put(18,45){\contour{white}{$f_2$}}
			\put(15,28){\contour{white}{$f_3$}}
			\put(27,21){\contour{white}{$f_4$}}
			\put(50,18){\contour{white}{$f_5$}}
			\put(66,20){\contour{white}{$f_6$}}
		\end{overpic}
		\relax\\
		\begin{overpic}[height=0.3\textwidth]{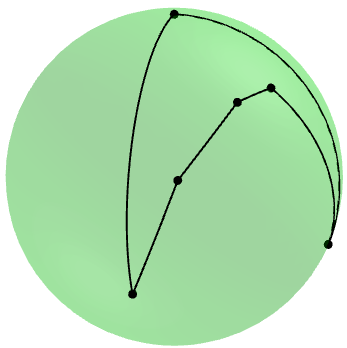}
		  \cput(38,11){{$\nw_{1}$}}
			\cput(44,95){{$\nw_{2}$}}
			\cput(95,25){{$\nw_{3}$}}
			\cput(80,67){{$\nw_{4}$}}
			\cput(70,64){{$\nw_{5}$}}
			\cput(54,43){{$\nw_{6}$}}	
		\end{overpic}
	}
	\centerline{
		\begin{overpic}[height=0.3\textwidth]{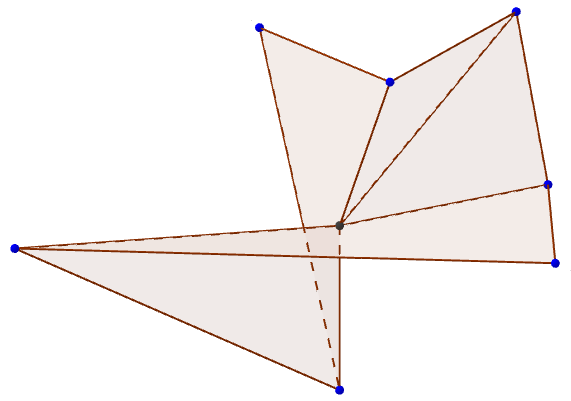}
		  \put(45,16){\contour{white}{$f_1$}}
			\put(75,28){\contour{white}{$f_2$}}
			\put(76,43){\contour{white}{$f_3$}}
			\put(70,51){\contour{white}{$f_4$}}
			\put(55,45){\contour{white}{$f_5$}}
		\end{overpic}
		\relax\\
		\begin{overpic}[height=0.3\textwidth]{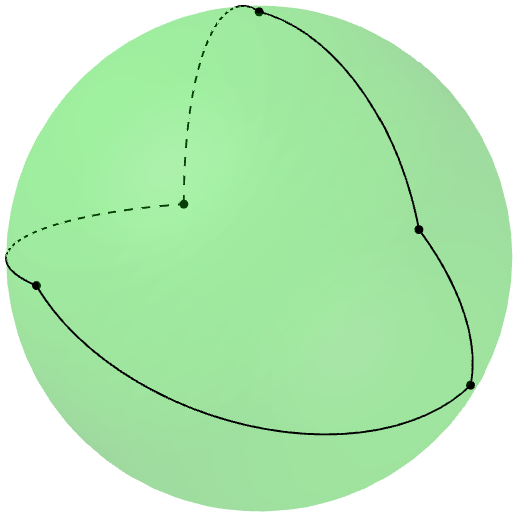}
		  \cput(40,57){{$\nw_{1}$}}
			\cput(46,93){{$\nw_{2}$}}
			\cput(86,53){{$\nw_{3}$}}
			\cput(93,20){{$\nw_{4}$}}
			\cput(12,43){{$\nw_{5}$}}
		\end{overpic}
	}
	\caption{Further cases of embedded vertex stars with Gauss images that are free of self-intersections}\label{fig:additional}
\end{figure}

We summarize our results in the following theorem.

\begin{theorem}\label{th:shape2}
Let ${\bf{v}}$ be a vertex of an embedded polyhedral surface and assume that $g({\bf{v}})$ has no self-intersections.

\begin{enumerate}
\item If $K({\bf{v}})> 0$, then $g({\bf{v}})$ is a convex spherical polygon.

\item If $K({\bf{v}})< 0$ and no face $f$ has a reflex angle at ${\bf{v}}$, then $g({\bf{v}})$ is a spherical pseudo-quadrilateral. Its four corners are the normals to the exactly four inflection faces in the star of ${\bf{v}}$.

\item If $K({\bf{v}})< 0$ and exactly one face $f$ has a reflex angle at ${\bf{v}}$, then $g({\bf{v}})$ is a spherical pseudo-triangle. In the case that $f$ is an inflection face, the three corners of $g({\bf{v}})$ are the normals of the three other inflection faces. If this $f$ is not an inflection face, its normal is a corner of $g({\bf{v}})$ and there are just two inflection faces in the star of ${\bf{v}}$, whose normals are the other two corners of $g({\bf{v}})$.

\item If $K({\bf{v}})< 0$ and more than one face has a reflex angle at ${\bf{v}}$, then $g({\bf{v}})$ is a spherical pseudo-digon. There are exactly two faces with a reflex angle at ${\bf{v}}$ and these two faces are inflection faces. The two corners of $g({\bf{v}})$ are the normals of the other two inflection faces in the vertex star.
\end{enumerate}
\end{theorem}

Polyhedral vertex stars whose Gauss images have no self-intersections play a key role in the recent paper \cite{GJP16}, where suitable assessments of smoothness of polyhedral surfaces are discussed. The property the authors start with is exactly that the Gauss image of the star of a vertex of either positive or negative discrete Gaussian curvature shall have no self-intersections. They also get the different cases illustrated in Figures~\ref{fig:convex}, \ref{fig:Gaussian}, and ~\ref{fig:additional}. Note that they use Theorem~\ref{th:egregium} and an area calculation using the embeddedness of the Gauss image to deduce Theorem~\ref{th:shape2}.

All the just mentioned cases beside the case with two reflex angles in the vertex star occurred in the paper \cite{ORSSSW04} of Orden et al. There, the authors investigated planar embeddings of planar frameworks and their reciprocals. The forces of a self-stress on the planar graph correspond to the dihedral angles of a possibly self-intersecting spherical polyhedron that projects to the graph. Sign changes of forces hence correspond to inflection faces. Considering polarity with respect to the paraboloid $z=x^2+y^2$ then leads to the reciprocal. Orden et al. were interested in the case when both the original framework and its reciprocal are crossing-free, translating to non-self-intersecting Gauss images in our setting. The requirement that the spherical polyhedron projects to the graph translates to the existence of a plane onto which the Gauss image projects in a one-to-one way, i.e., the existence of a transverse plane. In contrast to the authors of \cite{ORSSSW04}, we used arguments of spherical geometry only and did not use a projective dualization argument. That is also why the existence of a transverse plane was not necessary for our argument.


\section*{Acknowledgment}

The authors thank Wolfgang K\"uhnel for interesting discussions and Lara Skuppin for pointing out the relation between deformations without folding edges and the Whitney-Graustein theorem in the planar case. The paper was written while the second author was affiliated with Technische Universit\"at Berlin.

\bibliographystyle{plain}
\bibliography{Discrete_Gauss}

\begin{thebibliography}{10}

\bibitem{A05}
A.D. Alexandrov.
\newblock {\em Convex polyhedra}.
\newblock Springer Monographs in Mathematics. Springer, Berlin, 2005.

\bibitem{A88}
V.I. Arnold.
\newblock Ramified covering $\mathds{C}{P}^2 \to {S}^4$, hyperbolicity and
  projective topology.
\newblock {\em Siberian Math. J.}, 29(5):36--47, 1988.

\bibitem{B67}
T.F. Banchoff.
\newblock Critical points and curvature for embedded polyhedra.
\newblock {\em J. Diff. Geom.}, 1:245--256, 1967.

\bibitem{B70}
T.F. Banchoff.
\newblock Critical points and curvature for embedded polyhedral surfaces.
\newblock {\em Amer. Math. Monthly}, 77:475--485, 1970.

\bibitem{BF75}
T.F. Banchoff and F.~Takens.
\newblock Height functions on surfaces with three critical points.
\newblock {\em Illinois J. Math.}, 19(3):325--335, 1975.

\bibitem{BK82}
U.~Brehm and W.~K\"uhnel.
\newblock Smooth approximation of polyhedral surfaces regarding curvature.
\newblock {\em Geom. Dedicata}, 12:435--461, 1982.

\bibitem{C49}
A.~Cs\'asz\'ar.
\newblock A polyhedron without diagonals.
\newblock {\em Acta Sci. Math. Szeged.}, 13:140--142, 1949.

\bibitem{DK00}
V.~Dergachev and A.~Kirillov.
\newblock Index of {L}ie algebras of seaweed type.
\newblock {\em J. Lie Th.}, 10(2):331--343, 2000.

\bibitem{FGG95}
P.~Di Francesco, O.~Golinelli, and E.~Guitter.
\newblock Meander, folding and arch statistics.
\newblock {\em Math. Comput. Model.}, 26:97--147, 1995.

\bibitem{GJP16}
F.~G\"unther, C.~Jiang, and H.~Pottmann.
\newblock Smooth polyhedral surfaces.
\newblock ar{X}iv:1703.05318, 2017.

\bibitem{HT11}
Ch.E. Heitsch and P.~Tetali.
\newblock Meander graphs.
\newblock {\em Discr. Math. \& Theor. Comp. Sc.}, DMTCS {P}roceedings vol.
  {AO}, 23rd {I}nternational {C}onference on {F}ormal {P}ower {S}eries and
  {A}lgebraic {C}ombinatorics {(FPSAC 2011)}:12 pages, 2011.

\bibitem{JGWP16}
C.~Jiang, F.~G{\"u}nther, J.~Wallner, and H.~Pottmann.
\newblock Measuring and controlling fairness of triangulations.
\newblock In S.~Adriaenssens, F.~Gramazio, M.~Kohler, A.~Menges, and M.~Pauly,
  editors, {\em Advances in Architectural Geometry 2016}, pages 24--39. VDF
  Hochschulverlag, ETH Z{\"u}rich, 2016.

\bibitem{Lu69}
W.F. Lunnon.
\newblock A map-folding problem.
\newblock {\em Math. Comp.}, 22(3):193--199, 1968.

\bibitem{ORSSSW04}
D.~Orden, G.~Rote, F.~Santos, B.~Servatius, H.~Servatius, and W.~Whiteley.
\newblock Non-crossing frameworks with non-crossing reciprocals.
\newblock {\em Discr. Comput. Geom.}, 32(4):567--600, 2004.

\bibitem{P54}
G.~P\'olya.
\newblock An elementary analogue to the gauss-bonnet theorem.
\newblock {\em Amer. Math. Monthly}, 61:601--603, 1954.

\bibitem{ST38}
H.~Seifert and W.Threlfall.
\newblock {\em Variationsrechnung im {G}ro{\ss}en}.
\newblock Teubner Verlag, Berlin, 1938.

\bibitem{S58}
S.~Smale.
\newblock Regular curves on {R}iemannian manifolds.
\newblock {\em Trans. Amer. Math. Soc.}, 87:492--512, 1958.

\bibitem{W37}
H.~Whitney.
\newblock On regular closed curves in the plane.
\newblock {\em Compositio Math.}, 4:276--284, 1937.

\end{thebibliography}

\end{document}